\documentclass[11pt]{article}
\usepackage{bbm}
\usepackage{geometry}
\usepackage{color}
\usepackage{amsfonts}
\usepackage{latexsym, amssymb, amsmath, amscd, amsthm, amsxtra}
\usepackage{mathtools}
\usepackage{enumerate}
\usepackage[all]{xy}
\usepackage{mathrsfs}
\usepackage{fancyhdr}
\usepackage{listings}
\usepackage{hyperref}

\def\P{\mathbb{P}}
\def\E{\mathbb{E}}

\newtheorem{thm}{Theorem}[section]

\newtheorem{proposition}[thm]{Proposition}

\newtheorem{lemma}[thm]{Lemma}

\newtheorem{defn}[thm]{Definition}

\theoremstyle{definition}
\newtheorem{remark}[thm]{Remark}

\numberwithin{equation}{section}

\begin{document}
\title{Shotgun threshold for sparse Erd\H{o}s-R\'enyi graphs}

\author{Jian Ding\\Peking University \and Yiyang Jiang\\Peking University \and Heng Ma\\Peking University}

\maketitle

\begin{abstract}
In the shotgun assembly problem for a graph, we are given the empirical profile for rooted neighborhoods of depth $r$ (up to isomorphism) for some $r\geq 1$ and we wish to recover the underlying graph up to isomorphism. When the underlying graph is an Erd\H{o}s-R\'enyi $\mathcal G(n, \frac{\lambda}{n})$, we show that the shotgun assembly threshold $r_* \approx \frac{ \log n}{\log (\lambda^2 \gamma_\lambda)^{-1}}$ where $\gamma_\lambda$ is the probability for two independent Poisson-Galton-Watson trees with parameter $\lambda$ to be rooted isomorphic with each other. Our result sharpens a constant factor in a previous work by Mossel and Ross (2019) and thus solves a question therein.
\end{abstract}

\section{Introduction}

 Aiming for recovering a global structure from local observations, the shotgun assembly problems  have substantial interests in applications such as DNA sequencing \cite{AMRW96, DFS94, MBT13} and recovering neural networks \cite{KPPSP13}. In \cite{MR19}, the precise formulation and general mathematical framework was proposed together with numerous inspiring open questions on a number of concrete shotgun models. Since (the circulation of) \cite{MR19}, there has been extensive study on shotgun assembly questions including on random jigsaw problems \cite{BBN17, Martinsson19, BFM20,  HLW20}, on random coloring models \cite{PRS22}, on some extension of DNA sequencing models \cite{RBM21} and on lattice labeling models \cite{DL22}.

An example of shotgun assembling problems of particular interest is for random graph models \cite{MS15, GM22, HT21, AC22}. For random regular graphs with fixed degree, the asymptotic shotgun threshold was implied in \cite{Bollobas82} and was improved in \cite{MS15} to the precision of up to additive constant. For Erd\H{o}s-R\'enyi graphs with polynomially growing average degree, it was determined in \cite{GM22, HT21} whether recovery is possible from neighborhoods of depth $1$. 
For Erd\H{o}s-R\'enyi graphs with constant average degree, the shotgun threshold was known to have order $\log n$ from \cite{MR19}, and the main contribution of this paper is to determine its sharp asymptotics. (In fact, order $\log n$ was obtained except for the critical case, i.e, when the average degree is 1, where only a polynomial upper bound was obtained in \cite{MR19}.)

Before stating our result, we first define our model more formally. Fix $\lambda > 0$. Let $\mathcal G \sim \mathcal G(n, \frac{\lambda}{n})$, i.e., let $\mathcal G$ be an  Erd\H{o}s-R\'enyi graph on $n$ vertices where there is an edge between each unordered pair of vertices independently with  probability $\lambda/n$. For $v\in \mathcal G$ and $r\geq 1$, let $\mathsf{N}_{r}(v)$ be the depth $r$-neighborhood  rooted at $v$ viewed modulo isomorphism (equivalently, all other vertices in $\mathsf{N}_{r}(v)$ except $v$ are unlabeled). A couple of comments are in order: (1) here $\mathsf{N}_{r}(v)$ contains all vertices whose graph distances to $v$ are at most $r$ and all edges among these vertices; (2) here the notion of isomorphism for rooted graphs is as follows: a graph $G = (V, E)$ rooted at $o$ is isomorphic to a graph $G' = (V', E')$ rooted at $o'$ (denoted as $G\sim G'$) if there exists a bijection $\phi: V\mapsto V'$ with $\phi(o) = o'$ such that $(u, v)\in E$ if and only if $(\phi(u), \phi(v)) \in E'$. In the shotgun assembly problem, we are given the empirical profile for all rooted depth $r$-neighborhoods, i.e., we are given $\{\mathsf{N}_{r}(v): v\in \mathcal G\}$ and we wish to recover $\mathcal G$ up to isomorphism. We say the problem is \emph{non-identifiable} if there exist two graphs which are not isomorphic and both have empirical neighborhood profile $\{\mathsf{N}_{r}(v): v\in \mathcal G\}$; otherwise we say the problem is \emph{identifiable}. 

\begin{thm}\label{thm-main}
Fix $\lambda >0$. Define
\begin{equation}\label{eq-def-gamma}
\gamma_\lambda = \P( \mathbf{T} \sim\mathbf{T}')
\end{equation}
where $\mathbf{T}, \mathbf{T}'$ are two independent Poisson-Galton-Watson trees with parameter $\lambda$ (where the root of the tree is naturally the initial ancestor). Let $\mathcal G \sim \mathcal G(n, \frac{\lambda}{n})$. Then the following hold for any fixed $\epsilon_0>0$:
\begin{enumerate}[(i)]
   \item  for $r \leq \frac{(1-\epsilon_0) \log n}{ \log (\lambda^2 \gamma_\lambda)^{-1}}$, the shotgun problem is non-identifiable with probability tending to 1 as $n\to \infty$;
   \item  for $r \geq \frac{(1+\epsilon_0) \log n}{ \log (\lambda^2 \gamma_\lambda)^{-1}}$, the shotgun problem is identifiable with probability tending to 1 as $n\to \infty$ and in addition $\mathcal G$ can be recovered via a polynomial time algorithm. 
\end{enumerate}
\end{thm}

\begin{remark}\label{rem-alpha-lambda-monotonicity}
The function  $\lambda \mapsto (\lambda^2 \gamma_{\lambda}) $ is continuous on $(0,\infty)$, and is  increasing on $(0,1]$ as well as decreasing on $[1,\infty)$. See Lemma \ref{lem-alpha-lambda}.
\end{remark}

From the statement of Theorem~\ref{thm-main}, we see that a key novelty in this work is a connection to the isomorphic probability for Poisson-Galton-Watson (PGW) trees (isomorphism between random trees has been recently studied in \cite{Olsson22} although the isomorphic probabilities considered in \cite{Olsson22} are different from what we need here).
Therefore, the driving mechanism for the shotgun threshold of random regular graphs and Erd\H{o}s-R\'enyi graphs is completely different: as argued in \cite{MS15}, for random regular graphs ``tree neighborhoods are all alike; but every non-tree neighborhood is filled with cycles in its
own way'' and as a result cycle structures are essential in distinguishing neighborhoods in regular graphs; in the contrast, for Erd\H{o}s-R\'enyi graphs, local neighborhoods behave like PGW trees and loosely speaking these trees will all look different from each other with a suitably chosen depth.

That being said, cycles do appear in some of the neighborhoods of Erd\H{o}s-R\'enyi graphs and potentially this may incur an issue for our scheme of approximation by trees. The good news is that, by \eqref{eq-alpha-lambda-bound} a typical neighborhood with depth near the threshold is a tree so that conceptually our reduction to isomorphism between PGW trees is valid. However, on the technical side, the occurrence of cycles forms a substantial obstacle for the analysis which is why our proof for identifiability is fairly involved. For non-identifiability, while we do investigate neighborhoods with depth twice the critical threshold, we focus on some special type of structures (see, e.g., \eqref{eq-def-mathfrak-p-r-L}) which are trees.

Finally, it is non-trivial how exactly the isomorphic probability is related to the shotgun threshold.  As we will see in Section~\ref{sec:tree-isomorphic}, what is of fundamental importance to us is the probability that two independent PGW trees survive for at least $r$ levels and at the same time their first $r$ levels are isomorphic with each other (see \eqref{eq-def-frak-p-r}): this probability decays exponentially in $r$ where the exponential rate is governed by $\gamma_\lambda$ (see Lemma~\ref{Decay pr}). In order to further elaborate this connection, we feel it would be more useful to simply present the proof for non-identifiability (as incorporated in Section~\ref{sec:non-identifiability}), where tree isomorphism plays a role in constructing blocking configurations (which is inspired by and also an improvement upon what was considered in \cite{MR19}). The proof for identifiability, as mentioned above, is substantially more challenging and will be carried out in Sections~\ref{sec:identifiability} and \ref{sec:ER-property}. 

\smallskip

\noindent{\bf Notation convention.}  We denote by $\mathbb N$ the collection of all natural numbers. 
We use \emph{with high probability} for with probability tending to 1 as $n\to \infty$. For non-negative sequences $f_n$ and $g_n$, we write $f_n \lesssim g_n$ if there exists a constant $C>0$ such that $f_n \leq C g_n$ for all $n\geq 1$. We write $f_n \lesssim_\lambda g_n$ in order to stress that the constant $C$ depends on $\lambda$. We use $\mathrm{Bin}(n, p)$ to denote a binomial distribution/variable with parameter $(n, p)$ and we write $X\sim \mathrm{Bin}(n, p)$ if $X$ is a binomial variable with parameter $(n, p)$. As we will reiterate in the main text, we assume that there is a pre-fixed (but arbitrarily chosen) ordering on $V=V(\mathcal{G})$. We write $\mathcal G_{uv} = 1$ if and only if there is an edge in $\mathcal G$ between $u$ and $v$. For a graph $G$, we let $\mathrm{Comp}(G)$ be the \emph{complexity}, that is, $\mathrm{Comp}(G)$ is the minimal number of edges that one has to remove from $G$ so that no cycle remains. For a rooted tree $T$, we say it survives $\ell$ levels if its $\ell$-th level contains at least one vertex. For a rooted depth $r$-neighborhood $\mathsf N_r(v)$, we say $\mathsf N_r(v)$ survives if $\mathsf N_{r-1}(v) \neq \mathsf N_r(v)$.
 Also, we write depth $r$-neighborhood as $r$-neighborhood for short.

\smallskip

\noindent {\bf Acknowledgement.} We warmly thank Hang Du, Haojie Hou, Haoyu Liu, Yanxia Ren and Fan Yang for helpful discussions.

\section{Proof of non-identifiability}\label{sec:non-identifiability}

As mentioned earlier, we need a particular version of isomorphic probability which we now introduce. Along the way, we also introduce some useful notations.
 \begin{itemize}
 \item  For a rooted tree $T$, we denote by $|T|$ the size of $T$, and by $H(T)$ the height  of  tree $T$. For each positive integer $r$, let $T|_{r}$ be the restriction of the tree $T$ to its first $r$ levels. For an individual  $u \in T$, we denote by $|u|$ the level of $u$, i.e., the distance from $u$ to the root of $T$.

 \item  Let $T, T^{\prime}$ be  two rooted trees. Recall that we have defined isomorphism between $T$ and $T'$. In addition, for each $r \in \mathbb{N} \cup \{ \infty \}$, we write $T \sim_{r} T'$ if $\mathrm{min}\{ H(T), H(T')\} \geq r$ and  $T|_{r} \sim T'|_{r}$; this is the key notion of isomorphism for our analysis later.

 \item  Let $\mathbf{T}$ and $\mathbf{T}'$ be two independent Poisson-Galton-Watson trees with parameter $\lambda$ (PGW$(\lambda)$-trees).  Define
 \begin{equation}\label{eq-def-frak-p-r}
   \mathfrak{p}_{r}=\mathfrak{p}_{r}(\lambda)=\mathbb{P}\left(\mathbf{T} \sim_{r} \mathbf{T}'\right) \,, \mbox{ for } r\in \mathbb{N}\cup\{\infty\}\,.
 \end{equation}
\end{itemize}
The importance of $\mathfrak p_r$ lies in the following fact: if there exist two isomorphic $2r$-neighborhoods which are disjoint trees, then there should also exist $v, u$ such that their $2r$-neighborhoods are two disjoint isomorphic trees with some decoration (see Figure~\ref{p1} where the line segment between $w$ and $w'$ in the figure is the decoration) and also with some additional structural properties (in fact, we also will need to adjust the value of $r$ slightly in order to pose the additional structural properties). In this case, we can then construct two non-isomorphic graphs which have the same empirical profile for $r$-neighborhoods.

\begin{figure}[htbp]
\centering
\includegraphics[scale=0.6]{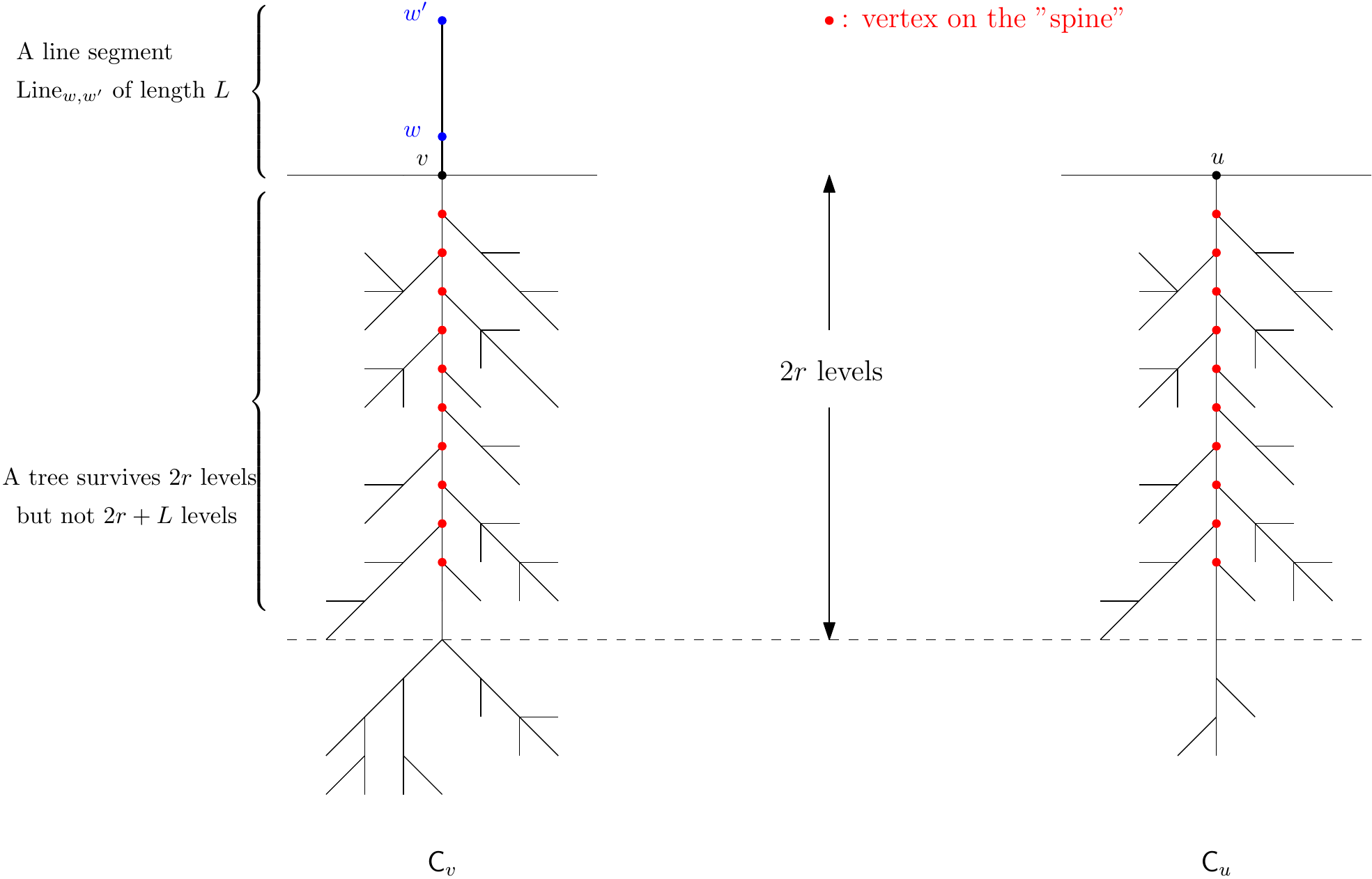}
\caption{Example of blocking subgraph} \label{p1}
\end{figure}

\subsection{Isomorphism for Galton-Watson trees}\label{sec:tree-isomorphic}

In this subsection 
we prove a number of lemmas on isomorphism of PGW-trees. Unless specified otherwise, we will denote by $\mathbf T, \mathbf T'$ two independent PGW($\lambda$) trees.

   \begin{lemma}\label{lem-monotone}
  The sequence $(\mathfrak{p}_{r})_{r \geq 1}$ is non-increasing in $r$, and  $ \mathfrak{p}_{r} \to \mathfrak{p}_{\infty}=0$ as $r\to \infty$.
   \end{lemma}
   \begin{proof}
 The monotonicity and convergence are obvious since $\{ \mathbf{T} \sim_{r+1} \mathbf{T}' \} \subset  \{ \mathbf{T} \sim_{r} \mathbf{T}' \}$ and $\cap_{r=1}^{\infty} \{ \mathbf{T} \sim_{r} \mathbf{T}' \}=\{ \mathbf{T} \sim_{\infty} \mathbf{T}' \}$. It remains to show that $\mathfrak{p}_{\infty}=0$.

 Let $Z_{n}$ and $Z_{n}'$ be the numbers of the vertices in the level $n$ of the tree $\mathbf{T}$, $\mathbf{T}'$ respectively. Then applying the Kesten-Stigum theorem  \cite{KS66} (see also \cite{LPP95})  we have that
    \begin{equation*}
 \P(\mathbf{T} \sim \mathbf{T}' \mbox{ and }  |H(\mathbf{T})| = |H(\mathbf{T}')| = \infty) \leq \P(Z_{n}= Z_{n}' \geq 1  ,  \forall \, n\geq 1 )  \leq \P(W= W'  > 0 )\,, 
    \end{equation*}
where $W$, $W'$ are the limits of $L^2$-bounded martingales $Z_{n}/ \lambda^n$ and $Z_{n}'/ \lambda^n$.
By  \cite{Harris48} (see also \cite[Theorem 8.3]{Harris63} and \cite{Stigum66}) $W$ has a probability density on the set $\{W>0\}$. Since $W$ and $W'$ are independent, we get that $\P(W= W'  > 0 ) = 0$.
\end{proof}

We next estimate the decay rate for $\mathfrak p_r$. To this end, define
\begin{equation}\label{eq-alpha-q}
\alpha_{\lambda}  = \lambda^2 \gamma_{\lambda} \mbox{ and } q_{\lambda}=\P(|\mathbf T| < \infty)\,.
\end{equation}
We claim that
\begin{equation}\label{eq-q-bound}
q_\lambda < \lambda^{-2} \mbox{ for } \lambda>1\,.
\end{equation}
There is nothing original in \eqref{eq-q-bound} and we supply a proof merely for completeness. It is well-known (by the method of conditioning on the number of children for the root) that $q_\lambda$ is the minimal zero for the equation $x = e^{-\lambda + \lambda x}$. Therefore, \eqref{eq-q-bound} can be reduced to $\exp(-\lambda + \lambda^{-1}) < \lambda ^{-2}$. Let $f(\lambda)= \lambda^{2}\exp(-\lambda + \lambda^{-1})$, then 
\begin{equation*}
f'(\lambda)= - (\lambda-1)^2 \exp(-\lambda + \lambda^{-1}) <0 \,, \text{ for all }\lambda > 1\,.
\end{equation*}
Thus $f(\lambda) < f(1)=1$ for all $\lambda > 1$, completing the verification.
The following lemma will be useful in controlling $\mathfrak p_r$ in Lemma~\ref{Decay pr}.

\begin{lemma}\label{lem-alpha-lambda}
For any $\lambda > 0$, we have $\alpha_{\lambda}= \alpha_{ \lambda q_{\lambda}} < 1$. Furthermore, there is a power series $A$ with non-negative coefficients such  that 
 $\alpha_{\lambda} = A( \lambda e^{-\lambda})$.
\end{lemma}
   \begin{proof}
  Since $ \P( \mathbf{T} \sim_{\infty} \mathbf{T}' )=0$ (by Lemma~\ref{lem-monotone}), we have
   \begin{equation*}
     \gamma_{\lambda}= \P(\mathbf{T} \sim \mathbf{T}'  )
      = \P(\mathbf{T} \sim \mathbf{T}'  \,\big|\,  |\mathbf{T}|, |\mathbf{T}'| < \infty  ) q_{\lambda}^2\,.
   \end{equation*}
It is well-known that the law of $\mathbf{T}$ under $\P( \cdot  \,\big|\,  |\mathbf{T}| < \infty)$ is the same as the law for PGW$(\lambda q_{\lambda})$.  
 Since $\mathbf{T}, \mathbf{T}'$ are independent, we get  $\P(\mathbf{T} \sim \mathbf{T}'  \,\big|\,  |\mathbf{T}|, |\mathbf{T}'| < \infty )=  \gamma(\lambda q_{\lambda} )$. This implies that  $\gamma (\lambda ) = \gamma(\lambda q_{\lambda} )  q_{\lambda}^2 $ and $ \alpha_{\lambda}= \alpha_{ \lambda q_{\lambda}}$.
In addition,  note $\gamma_{\lambda} \in (0,1)$   for all $\lambda >0$. Thus, $\alpha_{\lambda}<1$ for all $\lambda \leq 1$. For $\lambda>1$, thanks to \eqref{eq-q-bound} we have 
   $ \gamma_{\lambda} < q_{\lambda}^2  \leq \frac{1}{\lambda^4} $, implying $\alpha_\lambda<1$.

We now prove the second assertion. By Lemma~\ref{lem-monotone}, we have $\gamma_{\lambda}= \sum_{\tau} \P(\mathbf{T}  \sim \tau ) ^2$, 
where $\tau$ is summed over all the  equivalent classes for finite rooted ordered trees and the equivalence is given by rooted isomorphism. We set  $b(\tau)= \prod_{j=1}^{|\tau|}  b_{j}(\tau)!$, where    $\{ b_{j}(\tau):j=1, \ldots, |\tau| \}$ is the collection of the numbers of children for vertices in the tree $\tau$. 
 Then we have
\begin{equation*}
  \P(\mathbf{T}  \sim \tau)=  \frac{ \# \tau  }{b(\tau) }  e^{-\lambda |\tau|} \lambda^{|\tau|-1} \,,
\end{equation*}
where $\# \tau = \# \{ T \text{ is a rooted ordered tree}: T \sim \tau \}$. Let
\begin{equation*}
  A(s) =  \sum_{\tau}  \left[   \frac{ \# \tau  }{b(\tau) }  s ^{ |\tau|} \right]^2  \text{ for all } s \in [0, e^{-1}]\,.
\end{equation*}
Since $A(e^{-1}) = \gamma_1 < 1$, we see that $A$ is well-defined on $[0,e^{-1}]$. Then a simple computation yields that $\alpha_{\lambda}= A(\lambda e^{-\lambda})$.
\end{proof}
By Lemma~\ref{lem-alpha-lambda} and the fact that $\lambda e^{-\lambda}$ is increasing in $\lambda$ on $[0, 1]$ and decreasing in $\lambda$ on $[1, \infty)$, we obtain the desired monotonicity of $\alpha_{\lambda}$ as in  Remark~\ref{rem-alpha-lambda-monotonicity}. In addition, for $\lambda > 1$, we have from Lemma~\ref{lem-alpha-lambda} and \eqref{eq-q-bound} that
 \begin{equation}\label{eq-alpha-lambda-bound}
 \frac{1}{\log (1/ \alpha_\lambda)}   < \frac{1}{2\log (\lambda)} \text{ for } \lambda > 1.
 \end{equation}
 This implies that when the depth $r$ is near the shotgun threshold, a typical $r$-neighborhood has at most $O(n^{1/2 - \delta})$ vertices for some $\delta>0$ and as a result is a tree (see Lemma \ref{lemma-complexity-of-Nr}).

 For convenience, we let $\mu_k$'s be Poisson probabilities given by
 \begin{equation}\label{eq-def-mu-k}
 \mu_k = \mu_k(\lambda) =  \frac{\lambda^k}{k!} e^{-\lambda}\text{ for } k\geq 0\,.
 \end{equation}
 The following simple identity will be used repeatedly in our proof:
 \begin{equation}\label{eq-relation-mu-k}
 k\mu_{k}=\lambda \mu_{k-1} \text{ for all } k\geq 1\,.
 \end{equation}
\begin{lemma}\label{Decay pr}
     We have  $\mathfrak{p}_{r} \asymp \alpha_{\lambda}^{ r}$, i.e., there exist two constants $c, C>0$ (possibly depending on $\lambda$) such that  $c \alpha_{\lambda}^{r} \leq \mathfrak{p}_{r} \leq C \alpha_{\lambda}^{r}   $.
   \end{lemma}

 \begin{proof}
 For two independent PGW($\lambda$) trees $\mathbf T$ and $\mathbf T'$, define
\begin{equation}\label{eq-def-g-r}
 g_{r} = \P( \mathbf{T}|_{r} \sim \mathbf{T}'|_{r})\,.
\end{equation}
Note that
     \begin{equation*}
      \{ \mathbf{T}  \sim \mathbf{T}'  \} \subset     \{ \mathbf{T}|_{r} \sim \mathbf{T}'|_{r} \}     = \{  \mathbf{T}  \sim \mathbf{T}' , H(\mathbf{T}) \leq r \} \cup   \{  \mathbf{T}  \sim_{r} \mathbf{T}' , H(\mathbf{T}) > r \}\,.
     \end{equation*}
    Thus,  by Lemma~\ref{lem-monotone} we have
    \begin{equation}\label{eq-g-r-bound}
    \gamma_\lambda \leq  g_{r} \leq \gamma_\lambda + \mathfrak{p}_{r} \mbox{ and } g_{r} \downarrow \gamma_\lambda \mbox{ as } r \to \infty\,.
   \end{equation}

  Let $D, D'$ be the numbers of children for the roots of $\mathbf{T},\mathbf{T}'$ respectively. Let $\mathbf T_i$ be the subtree rooted at the $i$-th child for the root of $\mathbf T$ (similar notation applies for $\mathbf T'_i$), that is, the tree that consists of the $i$-th child as well as all of its descendants. Then, on the event $ \{D = D' = k\} $ we have
     \begin{equation}\label{eq-isomorphism-induction}
     \{ \mathbf{T} \sim_{r} \mathbf{T}^{\prime} \} =   \bigcup_{ 1 \leq i, j \leq k} \{  \mathbf{T}_{i} \sim_{r-1} \mathbf{T}_{j}^{\prime} \} \cap \{ (\mathbf{T} \backslash \mathbf{T}_{i})|_{r}  {\sim} (\mathbf{T}' \backslash \mathbf{T}'_{j} )|_{r}  \} \,.
   \end{equation}
On the one hand,  we get from \eqref{eq-isomorphism-induction} and \eqref{eq-relation-mu-k} that
  \begin{align}\label{eq-mathfrak-p-r-upper-bound}
      \mathfrak{p}_{r} & \leq \sum_{k=1}^{\infty} \mu_{k}^{2} k^{2} \mathfrak{p}_{r-1} \mathbb{P}\left(\mathbf{T}|_{r} \sim \mathbf{T}^{\prime}|_{r} \mid D=D^{\prime}=k-1\right) \nonumber \\
      &=\lambda^{2} \mathfrak{p}_{r-1} \sum_{k=1}^{\infty} \mu_{k-1}^{2} \mathbb{P}\left(\mathbf{T}|_{r} \sim \mathbf{T}^{\prime}|_{r} \mid D=D^{\prime}=k-1\right) =\lambda^{2} g_{r} \mathfrak{p}_{r-1}   \,.
 \end{align}
   On the other hand,  using the inequality  $\mathbb{P} ( \cup_{i} A_{i} ) \geq  \sum_{i} \mathbb{P}\left(A_{i}\right)-\frac{1}{2} \sum_{i \neq j} \mathbb{P}\left(A_{i} \cap A_{j}\right)  $ we get from  \eqref{eq-isomorphism-induction} that
      \begin{equation}\label{eq-mathfrak-p-r-lower-bound}
        \mathfrak{p}_{r}  \geq  \sum_{k =  1}^{\infty} k^{2} \mathfrak{p}_{r-1}
         \mathbb{P}\left( \mathbf{T}|_{r} \sim \mathbf{T}^{\prime}|_{r} \mid D=D^{\prime} =k-1 \right)   \mu_{k}^{2} - \frac{1}{2}\mathrm{err} = \lambda^2 g_{r} \mathfrak{p}_{r-1} - \frac{1}{2}\mathrm{err}
         \end{equation}
         where
         \begin{equation*}
       \mathrm{err} =
        \sum_{k = 2}^{\infty}
        \sum_{(i_{1},j_{1}) \neq (i_{2},j_{2}) }
        \P( \mathbf{T}_{i_{l}} \sim_{r-1} \mathbf{T}'_{j_{l}},  (\mathbf{T} \backslash \mathbf{T}_{i_{l}})|_{r} \sim (\mathbf{T}' \backslash  \mathbf{T}'_{j_{l} })|_{r} \mbox{ for } l \in \{1,2\} \mid D=D'=k  )  \mu_{k}^{2}\,.
      \end{equation*}
      Note that for $(i_{1},j_{1}) \neq (i_{2},j_{2})$, the event $\mathbf{T}_{i_{l}} \sim_{r-1} \mathbf{T}'_{j_{l}}$ and  $(\mathbf{T} \backslash \mathbf{T}_{i_{l}})|_{r} \sim (\mathbf{T}' \backslash  \mathbf{T}'_{j_{l} })|_{r} \mbox{ for } l \in \{1,2\} $ is contained in 
      $\cup_{s_1 \neq s_2, t_1 \neq t_2}\left\{ \mathbf{T}_{s_{1}} \sim_{r-1} \mathbf{T}'_{t_{1}}, \mathbf{T}_{s_{2}} \sim_{r-1} \mathbf{T}'_{t_{2}}  \right\}$.
Thus,  \begin{align*}
    \mathrm{err} & \leq  \sum_{k \geqslant 2} (k(k-1))^2 \mathfrak{p}_{r-1}^{2}   \mathbb{P}\left(\mathbf{T}|_{r} \sim \mathbf{T}^{\prime}|_{r} \mid D=D^{\prime}=k-2\right) \mu_{k}^{2}\\
    &  = \sum_{k \geqslant 2} \lambda^4 \mathfrak{p}_{r-1}^{2}   \mathbb{P}\left(\mathbf{T}|_{r} \sim \mathbf{T}^{\prime}|_{r} \mid D=D^{\prime}=k-2\right) \mu_{k-2}^{2} =  \lambda^4 g_r \mathfrak p_{r-1}^2
\end{align*}
 where we have used \eqref{eq-relation-mu-k} twice. Combined with \eqref{eq-mathfrak-p-r-upper-bound} and \eqref{eq-mathfrak-p-r-lower-bound}, it yields that
\begin{equation}\label{eq-mathfrak-p-r-inequality}
\lambda^2 g_{r} \mathfrak{p}_{r-1} -    \lambda^4 g_{r}  \mathfrak{p}_{r-1}^{2}  \leq \mathfrak{p}_{r}  \leq  \lambda^2 g_{r} \mathfrak{p}_{r-1}\,.
\end{equation}
 By \eqref{eq-g-r-bound} and $\mathfrak p_r \to 0$ (Lemma~\ref{lem-monotone}), we have that $\mathfrak{p}_{r} = \alpha_\lambda^{ (1+o(1)) r }$ where $o(1)$ vanishes in $r$. Since $\alpha_\lambda<1$ (Lemma~\ref{lem-alpha-lambda}), this implies that $\mathfrak p_r$ decays exponentially. Combining this with \eqref{eq-g-r-bound} and \eqref{eq-mathfrak-p-r-inequality}, we see that $\sum_{i=1}^\infty |\frac{\mathfrak p_{i+1}}{\mathfrak p_i} - \alpha_\lambda| < \infty$, yielding the desired bound.
 \end{proof}
In order to prove non-identifiability, we need to pose some additional structure on two isomorphic trees. To this end, we say a sequence of vertices
 $(v_{i})_{i=0}^{m}$ is a \textit{spine}
 of $T$ if it is a path (we say a sequence of distinct vertices is a path if each neighboring pair is connected by an edge) of $T$ started at $v_0$. Furthermore, for  $v \in T$, we denote by $T_{v}$ the   subtree rooted at $v$. For $r, L\geq 1$, define $\mathfrak p_{r, L} = \P(\mathcal E(\mathbf T, \mathbf T'; r, L))$ where $ \mathcal E( T, T'; r, L)$ is the event that
 \begin{align}\label{eq-def-mathfrak-p-r-L}
    \begin{split}
   & T \sim_{r} T' ,   H(T) , H(T')\leq r+ L ; H(T)  \neq  H(T');\\
   & \text{every vertex in $r'$-th level of $T, T'$ has either 0 or two children for all $r' > r$};\\
    & \exists \text{ spine } (v_{i})_{i=0}^{r-L} \text{ in } T, \text{ such that } |T_{v_{i-1}}   \backslash T_{v_{i}} | \leq L
      \text{ for all }  i\,.
      \end{split}
\end{align}

 \begin{lemma}\label{spine}
  For every $\epsilon > 0$, there exists $L=L_{\epsilon}$ 
depending only on $\epsilon$ such that $\mathfrak p_{r, L} \geq (\alpha_\lambda-\epsilon)^r$ for sufficiently large $r$.
   \end{lemma}
  \begin{proof}
     For two rooted trees $T$ and $T'$, define the event
    $A(T,T')=\{ T \sim T', |T| \leq L\}$.
  Note that $ \mathcal E( T, T'; r, L) $ can be defined inductively. For $r>L$, we have that $\mathcal E( T, T'; r, L)$ occurs if and only if there exists a \emph{unique} pair $(v_1, v_1')\in (T, T')$ such that $\mathcal E( T_{v_1}, T'_{v'_1}; r-1, L)$ occurs and that $A(T\setminus T_{v_1}, T' \setminus T'_{v'_1})$ occurs. Therefore, denoting by $D, D'$ the numbers of children for the roots of $\mathbf T, \mathbf T'$ and recalling \eqref{eq-def-mu-k}, we have that
  \begin{align}
  \mathfrak p_{r, L} & = \sum_{k = 1}^\infty \sum_{1\leq v_1, v'_1 \leq k} \mu_k^2 \P(\mathcal E(\mathbf T_{v_1}, \mathbf T'_{v'_1}; r-1, L)) \P(A(\mathbf T\setminus \mathbf T_{v_1}, \mathbf T' \setminus \mathbf T'_{v'_1}) \mid D = D' = k ) \nonumber\\
  & = \mathfrak p_{r-1, L} \lambda^2 \sum_{k = 1}^\infty \mu_{k-1}^2  \P(A(\mathbf T, \mathbf T') \mid D = D' = k-1 ) = \mathfrak p_{r-1, L} \lambda^2 \P(A(\mathbf T, \mathbf T'))\,, \label{eq-mathfrak-p-r-L-recursion}
  \end{align}
  where the second inequality follows from \eqref{eq-relation-mu-k} and the fact that the conditional law of $(\mathbf T\setminus \mathbf T_{v_1}, \mathbf T' \setminus \mathbf T'_{v'_1})$ given  $D = D' = k$ is the same as the conditional law of $(\mathbf T, \mathbf T')$ given $D = D' = k-1$. By Lemma~\ref{lem-monotone}, we see that $\P(A(\mathbf T, \mathbf T') \mid \mathbf T \sim \mathbf T') \to 1$ as $L \to \infty$. Therefore, for any $\epsilon>0$, there exists $L = L_{\epsilon}$ such that $\P(A(\mathbf T, \mathbf T')) \geq \gamma_\lambda - \epsilon/(2\lambda^2)$. Combined with \eqref{eq-mathfrak-p-r-L-recursion}, this gives that
  $$\mathfrak p_{r, L} \geq (\alpha_\lambda - \epsilon/2)^{r-L} \mathfrak p_{L, L} \geq c_L (\alpha_\lambda - \epsilon/2)^{r-L} $$ where $c_L>0$ depending only on $(L, \lambda)$. This completes the proof.
 \end{proof}

 We also need an estimate that compares local neighborhoods of Erd\H{o}s-R\'enyi graphs to PGW trees. This has been well-understood and fairly straightforward. For instance, a straightforward extension of \cite[Lemma 2.2]{RW10} leads to the following lemma (so we omit the proof). For each vertex $v\in \mathcal G$, let $\mathsf C_v$ be the component of $v$ in $\mathcal G$ with root $v$. 

 \begin{lemma}\label{lem-compare-extension}
 Let $\mathbf T_i$'s be independent PGW($\lambda$)-trees.
 For $k\geq 1$ and $v_1, \ldots, v_k \in \mathcal G$, and for rooted trees $\tau_1, \ldots, \tau_k$ with $\sum_{i=1}^k |\tau_i| = o(\sqrt{n})$, we have that as $n \to \infty$,
 \begin{equation*}
  \P(\mathsf C_{v_i} \sim \tau_i \text{ for } 1\leq i\leq k \text { and }   \mathsf{C}_{v_{i}} \text{'s are  disjoint} ) = (1 + o(1)) \P(\mathbf T_i \sim \tau_i \text{ for } 1\leq i\leq k)\,.
 \end{equation*} 
 \end{lemma}

\subsection{Proof of non-identifiability}

 In this section we prove the non-identifiability. We set $r= \frac{(1- \epsilon_0) \log n}{\log \alpha_\lambda^{-1}}$  so that
\begin{equation}\label{eq-def-r-non-identifiable}
\alpha_{\lambda}^{2r}  =  n^{-2(1-\epsilon_0)} \text { for an arbitrary fixed small } \epsilon_0>0 \,.
\end{equation}
Moreover,  we choose $\epsilon= \alpha_{\lambda}(1- \alpha_{\lambda}^{  2\epsilon_0/ [ 3(1-\epsilon_0)]} )  $   and $L=L_{\epsilon}$ in Lemma \ref{spine} so that 
\begin{equation}\label{eq-alpha-epsilon}
  (\alpha_\lambda- \epsilon)^{2r}  =  n^{-2(1-\epsilon_0/3)}   \,.
\end{equation}
In order to prove non-identifiability, we will construct a \emph{blocking configuration} as in \cite{MR19} whose existence certifies non-identifiability, and then we need to show that with high probability such a blocking configuration exists.

\noindent {\bf Construction of blocking configuration.}  We refer to Figure \ref{p1} for an illustration of the construction.  We say $(\mathsf C_v, \mathsf C_u)$ is a blocking configuration if it satisfies the following property:
\begin{itemize}
\item $\mathsf C_v, \mathsf C_u$ are disjoint trees and there is a line segment $\mathrm{Line}_{w, w'}$ of length $L$ with endpoints $w, w'$ and $w$ is connected to $v$.
\item The event $\mathcal E(\mathsf C_v \setminus \mathrm{Line}_{w, w'}, \mathsf C_u; 2r, L)$ holds where $\mathcal E$ is defined as in \eqref{eq-def-mathfrak-p-r-L}.
\end{itemize}
We next show that when a blocking configuration exists for some pair $(\mathsf C_v, \mathsf C_u)$, the graph is non-identifiable from the empirical profile for $r$-neighborhoods. Indeed, if we remove the edge $(w, v)$ and add an edge $(w, u)$, i.e., cut $\mathrm{Line}(w, w')$ from $v$ and attach it to $u$, then we claim that:
\begin{enumerate}[(i)]
\item the empirical profile for $r$-neighborhoods is unchanged;
\item the empirical profile for $(2(r+L))$-neighborhoods is changed.
\end{enumerate}
Assuming the claim, we see that the empirical profile of $r$-neighborhoods does not determine the whole graph up to isomorphism since it does not even determine a unique empirical profile for $(2(r+L))$-neighborhoods.

We now prove (i). Note that the `cut-attach' procedure only changes $r$-neighborhoods for vertices in $\mathsf N_r(v)$ and $\mathsf N_r(u)$. Since $(\mathsf C_v\setminus \mathrm{L}_{w, w'}) \sim_{2r} \mathsf C_u$, we can let $\phi$ be an isomorphism between these two trees. For a vertex $z$, we let $\widetilde {\mathsf N_r}(z)$  be the $r$-neighborhood of $z$ after the `cut-attach' procedure. Then it is clear that $\mathsf{N}_r(z)  \sim \widetilde {\mathsf N_r}(z) $ for all $z \in \mathrm{Line}_{w,w'}$,
$\mathsf{N}_r(v')\sim \widetilde {\mathsf N_r}(\phi(v'))$   for all  $v' \in \mathsf N_r(v)\backslash \mathrm{L}_{w, w'}$ and  $\mathsf{N}_r(u') \sim  \widetilde {\mathsf N_r}(\phi^{-1}(u'))$ for all $u' \in \mathsf N_r(u )$. This implies (i).

We next prove (ii). This is where we need the additional structure in the definition of \eqref{eq-def-mathfrak-p-r-L}. Suppose without loss of generality that $H(\mathsf C_v \setminus \mathrm{Line}_{w, w'}) > H(\mathsf C_u)$. By properties in \eqref{eq-def-mathfrak-p-r-L}, we see that the diameter of $\mathsf C_v$ is $L + H(\mathsf C_v \setminus \mathrm{Line}_{w, w'}) > L + H(\mathsf C_u)$ which is the diameter of $\mathsf C_u \cup \mathrm{Line}_{w, w'} \cup \{(u, w)\}$. This implies (ii), since after the `cut-attach' procedure the maximal diameter in the components of $u$ and $v$ is changed.

\noindent {\bf Existence of blocking configuration.}
For $v, u\in \mathcal G$,  let $X_{v,u}$ be the  indicator function that $(\mathsf C_v, \mathsf C_u)$ is a blocking configuration. Let $N = \sum_{v,u} X_{v,u}$.
We need to show that with high probability $N\geq 1$. To this end, we need some input from PGW trees.
Denote by $\mathrm{BConf}$ the collection of all pairs of rooted unlabeled trees which are blocking configurations. That is, $\mathrm{BConf}$ consists of the equivalent classes of blocking configurations and the equivalence is given by rooted isomorphism.
 Let $\mathbf T, \mathbf T'$ be two independent PGW($\lambda$)-trees. For $(\tau, \tau') \in \mathrm{BConf}$, the roots of $\tau, \tau'$ have degree at most $L$. As a result, under the law of PGW($\lambda$)-tree, cutting a line of size $L$ from the root of $\tau$ only changes its probability density up to a factor depending on $(L, \lambda)$. Therefore, by Lemma~\ref{spine}  and   \eqref{eq-alpha-epsilon} we have that for some constant $c = c(L, \lambda)$
\begin{equation}\label{eq-prob-BConf}
\sum_{(\tau, \tau')\in \mathrm{BConf}}\P(\mathbf T\sim \tau, \mathbf T'\sim \tau') \geq c (\alpha - \epsilon)^{2r} = c n^{-2(1 - \epsilon_0/3)}\,.
\end{equation}
We also note that $|\tau|, |\tau'| = O(\log n)$  whenever $(\tau,\tau') \in \mathrm{BConf}$ and we will use this fact repeatedly (e.g., together with Lemma~\ref{lem-compare-extension}).

We are now ready to employ the second moment method in order to show $N \geq 1$. We have
\begin{align*}
  \E [N^2] &=   \E [N] + \sum_{(u_1,u_2)\neq (u_3,u_4)} \E( X_{u_1,u_2} X_{u_{3},u_4} ) \\
  & \leq \E N+ n^{4} \E(X_{v_{1}, v_{2}}X_{v_{3}, v_{4} }) + n^{3} \E(X_{v_{1} , v_{2}}X_{v_{1}, v_{4}} )+ n^{3} \E(X_{v_{1}, v_{2}}X_{v_{3}, v_{2}} ) \,,
\end{align*}
where $v_{1}, \ldots, v_{4}$ are pairwise different. Thus, it suffices to show that
\begin{align*}
 & \text{ (a) } \lim_{n \to \infty} n^2  \E(X_{v_{1},v_{2}}) = \infty  \  ; \ \text{ (b) }
 \limsup_{n \to \infty} \frac{\E(X_{v_{1}, v_{2}}X_{v_{3}, v_{4} })}{ (\E X_{v_{1},v_{2}})^2  }
          \leq 1  \ ; \\
& \text{ (c) } \lim_{n \to \infty} \frac{ \E(X_{v_{1} , v_{2}}X_{v_{1}, v_{4}} ) }{ n (\E X_{v_{1},v_{2}})^2  } = 0 \text{ and } \lim_{n \to \infty} \frac{ \E(X_{v_{1}, v_{2}}X_{v_{3}, v_{2}} ) }{ n (\E X_{v_{1},v_{2}})^2  } = 0   \,.
\end{align*}
Indeed, having verified (a), (b) and (c),  we can then apply Cauchy-Schwarz inequality and get that  $\P(N \geq 1) \geq \frac{ [\E N]^2 }{ \E [ N^2]} \to 1 $ as $n \to \infty$.  In what follows we denote by $\mathbf T_i$'s as independent PGW($\lambda$)-trees.

\noindent \emph{Proof of (a)}. By Lemma~\ref{lem-compare-extension} and \eqref{eq-prob-BConf}, we get that
\begin{equation}\label{eq-N-exp-lower-bound}
\E X_{v_1, v_2} = (1+o(1)) \sum_{(\tau_1, \tau_2)\in \mathrm{BConf}}\P(\mathbf T_1\sim \tau_1, \mathbf T_2\sim \tau_2) \mbox{ and } \E N  \geq (c+o(1)) n^{2\epsilon_0/3}\,,
\end{equation}
which implies (a).

\noindent \emph{Proof of (b)}.  Write $\Omega$ as the event that $\mathsf C_{v_1}, \ldots, \mathsf C_{v_4}$ are mutually disjoint. By Lemma~\ref{lem-compare-extension} and \eqref{eq-prob-BConf}, we get that
\begin{align*}
  \E(X_{v_{1}, v_{2}}X_{v_{3}, v_{4} } \mathbf 1_{\Omega}) = (1+o(1)) \E (X_{v_1, v_2}) \E (X_{v_3,v_4}) = (1+o(1)) (\E X_{v_1, v_2})^2.
\end{align*}
Therefore, in order to prove (b) it suffices to show that
\begin{equation}\label{eq-bound-Omega-c}
  \E(X_{v_{1}, v_{2}}X_{v_{3}, v_{4} } \mathbf 1_{\Omega^c}) =  o((\E X_{v_1, v_2})^2).
\end{equation}
By the definition of blocking configuration (see properties in \eqref{eq-def-mathfrak-p-r-L}), for $(\tau_1, \tau_2) \in \mathrm{BConf}$, there is only one vertex $w'\in \tau_1 \cup \tau_2$ whose $L$-neighborhood is a line $\mathrm{Line}_{w, w'}$ of length $L$ (with $w'$ being one endpoint), and thus the root of $\tau_1$ is the only vertex  which is connected to $w$ (the other endpoint of $\mathrm{Line}_{w, w'}$).  Therefore, when   $  X_{v_{1}, v_{2}}X_{v_{3}, v_{4}} =1 $,  we  have $v_{3} \notin \mathsf{C}_{v_{1}} \cup \mathsf{C}_{v_{2}}$ and $v_{4} \notin \mathsf{C}_{v_{1}}$. Thus, if in addition $
\Omega$ does not occur, we must have that $ v_{4} \in \mathsf{C}_{v_{2}}$.
By Lemma~\ref{lem-compare-extension} and the fact that blocking configuration has size $O(\log n)$, we get that
\begin{align*}
  & \E[X_{v_{1}, v_{2}}X_{v_{3}, v_{4} } \mathbf 1_{ \Omega^{c} }  ] 
    \leq \sum_{ \substack{ (\tau_{1}, \tau_{2}) \in \mathrm{BConf} \\ (\tau_{3} , \tau_{4}) \in  \mathrm{BConf}   } }  \P [  \mathsf C_{v_{j}} \sim \tau_{j}, 1 \leq j \leq 4; \mathsf C_{v_1}, \mathsf C_{v_2}, \mathsf C_{v_3} \mbox{ disjoint } ; v_{4} \in \mathsf{C}_{v_{2}} ] \\
  &  =(1+o(1))  \sum_{   (\tau_{1}, \tau_{2}) \in \mathrm{BConf}  }  \P( \mathbf T_1 \sim \tau_{1}, \mathbf T_2 \sim \tau_{2} )  \sum_{  (\tau_{3} , \tau_{4}) \in  \mathrm{BConf}   }  \P( \mathbf T_3 \sim \tau_{3}) 1_{ \{\tau_{4} \in  \mathcal{E}(\tau_{2})\}} \frac{ O(\log n)  }{n} \,,
  \end{align*}
 where $\mathcal E(\tau_{2})$ consists of all (equivalent classes for) rooted trees that arise from $\tau_2$ by re-choosing the root as an arbitrary vertex in $\tau_2$, 
and $O(\log n/n)$ factor comes from the probability that $v_4 \in \mathsf C_{v_2}$.   
Noting that  $|\mathcal E(\tau_{2})| = O(\log n)$, we see that the number of valid choices for $\tau_3$ in the preceding sum is $O(\log n)$. Thus,
$$\E[X_{v_{1}, v_{2}}X_{v_{3}, v_{4} } \mathbf 1_{ \Omega^{c} }  ] \leq  \sum_{   (\tau_{1}, \tau_{2}) \in \mathrm{BConf}  }  \P( \mathbf T_1 \sim \tau_{1}, \mathbf T_2 \sim \tau_{2} )  \max_{  (\tau_{3} , \tau_{4}) \in  \mathrm{BConf}}  \P( \mathbf T_3 \sim \tau_{3}) \frac{ O((\log n)^2)  }{n}\,.$$
In addition,
\begin{equation}\label{eq-max-p-r-prob}
\max_{  (\tau_{3} , \tau_{4}) \in  \mathrm{BConf}}  \P( \mathbf T_3 \sim \tau_{3}) \leq \sqrt{\mathfrak p_{2r}} = O(\alpha_\lambda^{r})
\end{equation}
where the last equality follows from Lemma \ref{Decay pr}. Altogether, we get that
$$\E[X_{v_{1}, v_{2}}X_{v_{3}, v_{4} } \mathbf 1_{ \Omega^{c} }  ] \leq   \P( (\mathbf T_1 , \mathbf T_2) \in \mathrm{BConf}) \frac{ O((\log n)^2)  \alpha_\lambda^{r}}{n}\,.$$
Combined with \eqref{eq-def-r-non-identifiable} and \eqref{eq-N-exp-lower-bound}, this yields \eqref{eq-bound-Omega-c} as required.

 \noindent \emph{Proof of (c)}.  When $X_{v_{1} , v_{2}}X_{v_{1}, v_{4}} =1$, (using the uniqueness of the line graph of length $L$ in the blocking configuration as argued in Proof of (b)) we must have either $\mathsf{C}_{v_1}, \mathsf C_{v_2}, \mathsf C_{v_4}$ are disjoint (which we denote by the event $\tilde \Omega$) or $v_{4} \in \mathsf C_{v_2}$. By Lemma~\ref{lem-compare-extension}, we get 
  \begin{align*}
    \E[X_{v_{1}, v_{2}}X_{v_{1}, v_{4} } \mathbf 1_{\tilde \Omega}  ] & = (1+o(1) ) \sum_{ \substack{ (\tau_{1}, \tau_{2}) \in \mathrm{BConf} \\ (\tau_{1} , \tau_{4}) \in \mathrm{BConf}  } } \P (\mathbf T_j \sim \tau_{j} \mbox{ for } j =1,2,4) \\
    & = \sum_{ (\tau_{1}, \tau_{2}) \in \mathrm{BConf} } \P (\mathbf T_1 \sim \tau_{1}, \mathbf T_2 \sim \tau_2) \P(\mathbf T_4 \sim \mathbf \tau_2)
  \end{align*}
  where we have used the fact that the valid configurations in the summation satisfy $\tau_4 = \tau_2$. Therefore, combined with \eqref{eq-N-exp-lower-bound} and \eqref{eq-max-p-r-prob}, it yields that
$$ \E[X_{v_{1}, v_{2}}X_{v_{1}, v_{4} } \mathbf 1_{\tilde \Omega}  ] \leq \P((\mathbf T_1, \mathbf T_2) \in \mathrm{BConf}) \sqrt{\mathfrak p_{2r}} = o(n (\E X_{v_1, v_2})^2)\,.$$
We next estimate the expectation on the event $\tilde \Omega^c$. We have
  \begin{align*}
    \E[X_{v_{1}, v_{2}}X_{v_{1}, v_{4} } \mathbf 1_{ \tilde \Omega^{c} }  ]
   & \leq \sum_{  (\tau_{1}, \tau_{2}) \in \mathrm{BConf}    } \P (\mathbf T_1 \sim \tau_1, \mathbf T_2 \sim \tau_2) O(\frac{\log n}{n}) = o(n (\E X_{v_1, v_2})^2)\,,
     \end{align*}
     where the $O(\frac{\log n}{n})$ comes from the event that $v_4\in \mathsf C_{v_2}$ (and $|\mathsf C_{v_2}| = O(\log n)$ by definition of blocking configuration), and the last equality used \eqref{eq-N-exp-lower-bound}.

When $X_{v_{1}, v_{2}}X_{v_{3}, v_{2}}=1$, we must have that  $\mathsf{C}_{v_1}, \mathsf C_{v_2}, \mathsf C_{v_3}$ are disjoint (again using the uniqueness of the line graph of length $L$ in the blocking configuration). Then the bound on $\E (X_{v_{1}, v_{2}}X_{v_{3}, v_{2}})$ can be derived in the same way as that for $\E[X_{v_{1}, v_{2}}X_{v_{1}, v_{4} } \mathbf 1_{\tilde \Omega}  ]$, completing the proof of (c) and thus completing the proof of non-identifiability.

\section{Proof of identifiability}\label{sec:identifiability}

The rest of the paper is devoted to the proof of identifiability when 
\begin{equation}\label{eq-rho-r-identifiable}
r = \frac{(1+\epsilon_0) \log n}{\log \alpha_\lambda^{-1}} \,, \  \alpha_\lambda^r < \alpha_\lambda^{\rho r} \leq n^{-1 - \frac{\epsilon_0}{2}} \mbox{ and } \frac{\log(\alpha_{\lambda}^{-1}) }{ 2(1+\epsilon_0)} > \log(\lambda) \,,
\end{equation} 
for an arbitrarily fixed $\epsilon_0>0$ and  $\rho = \rho(\epsilon_0) < 1$. 
The third inequality is possible for small $\epsilon_0$ since
 $\frac{\log(\alpha_{\lambda}^{-1}) }{ 2 } > \log(\lambda) $ by \eqref{eq-alpha-lambda-bound} (note that identifiability is harder for smaller $r$, so it is fine to assume $\epsilon_0$ to be small as incorporated by the third inequality), and we made such assumption for convenience of controlling the volume of the $r$-neighborhood as in Lemma~\ref{lemma-complexity-of-Nr}. 
We now explain how to relate Lemma~\ref{Decay pr} to the above choice of $r$.  We say a neighborhood $\mathsf N_r(v)$ has \textit{two $r$-arms} (or we say $v$ has \textit{two $r$-arms}) if there are two  paths of length $r$ which are both rooted at $v$ and intersect only at $v$. 
 By Lemma~\ref{Decay pr}, we would \emph{essentially} get that the probability for two $r$-neighborhoods with two $r$-arms to be isomorphic is at most $o(n^{-2})$ and as a result \emph{essentially} each $r$-neighborhood with two arms is unique. (Here the word ``essentially'' refers to omitting the consideration for the scenario of $\mathsf N_r(v) \sim \mathsf N_r(u)$ when both neighborhoods have two $r$-arms but $u, v$ are contained in a short cycle; see Figure~\ref{Fig-example-cycle}.)
\emph{Intuitively}, this would be sufficient for recovery since a vertex $v$ is either contained in some $r$-neighborhood (not necessarily rooted at $v$) with two $r$-arms or the component of $v$ is contained in some $r$-neighborhood.

\begin{figure}[htbp]
  \centering \includegraphics[scale=0.9]{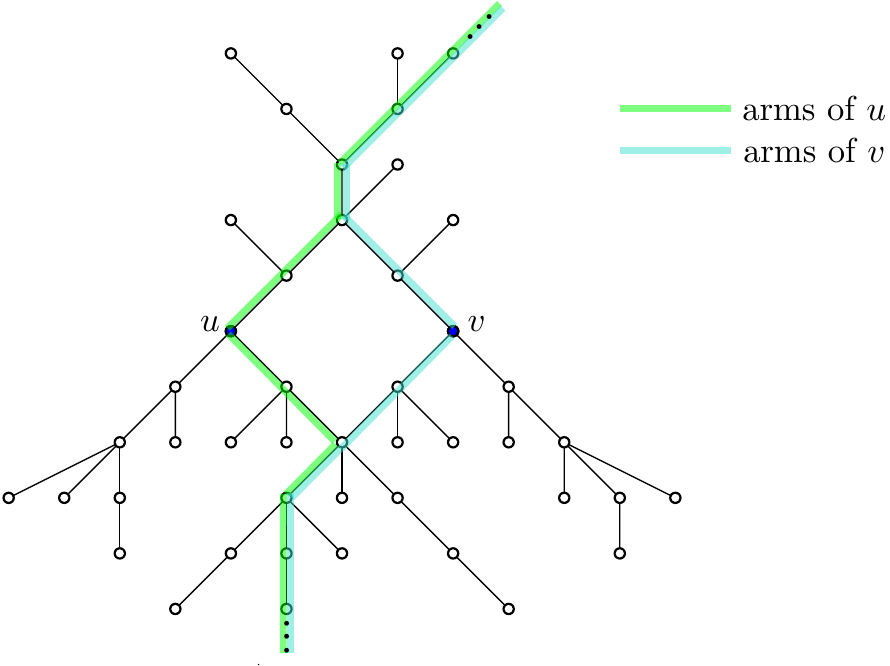}
\caption{
An illustration for  isomorphic neighborhoods
 with two arms centered on a  cycle
}
\label{Fig-example-cycle}
\end{figure}

We now describe how we treat small components as our  \emph{preprocessing} procedure. We say a vertex $v$ is \emph{degenerate} if $\mathsf N_r (v) = \mathsf N_{r-1}(v)$. Note that for a degenerate $v$, we have that
\begin{equation}\label{eq-degenerate-property}
\mathsf N_r(v) = \mathsf C_v \mbox{ and } \mathsf N_r(u) \sim \mathsf N_r(u; \mathsf C_v) \mbox{ for all } u\in \mathsf N_r(v)\,,
\end{equation}
where $\mathsf N_r(u; \mathsf C_v)$ is the $r$-neighborhood of $u$ in $\mathsf C_v$.
 Recall that we have assumed an ordering on $V$ (in the paragraph for notation convention).
 We set $U_1 = U_2 = \emptyset$ and iteratively apply the following procedure:
\begin{itemize}
\item If $V\setminus U_2 = \emptyset$, stop. Otherwise, take the minimal degenerate $v\in V\setminus U_2$, and choose a subset $A \subset V$ such that $\{\mathsf N_r(u; \mathsf{C}_v): u\in \mathsf C_v\} = \{\mathsf N_r(u): u\in A\}$ 
(this is possible due to \eqref{eq-degenerate-property});
\item Add vertex $v$ to $U_1$ and add vertices in $A$ to $U_2$.
\end{itemize}
At the end of our procedure, we have that $\cup_{v\in U_1} \mathsf C_v$ contains all degenerate vertices (it may also contain some non-degenerate vertices as well) and
$$\{\mathsf N_r(u): u\in \cup_{v\in U_1} \mathsf C_v\} = \{\mathsf N_r(u): u\in U_2\}\,.$$ Therefore, if we can identify the graph from $\{\mathsf N_r(u): u\in V\setminus U_2\}$ up to isomorphism, then adding disjoint components $\{\mathsf C_v: v\in U_1\}$ to it yields the original graph up to isomorphism. For this reason, in what follows we assume without loss of generality that all vertices are not degenerate (or equivalently, we have removed components of degenerate vertices using the preprocessing
procedure above).

We next describe our recovery procedure, whose success relies on certain structural properties for the Erd\H{o}s-R\'enyi graph which we discuss later. We say a vertex $v$ is \emph{good} if $\mathsf N_{\rho r}(v)$ is unique (among $\{\mathsf N_{\rho r}(u): u\in \mathcal G\}$), and we let $V_{\mathrm {g}}$ be the collection of all good vertices (crucially, a good vertex can be regarded as labeled since its neighborhood is unique).
Write $V_{\mathrm{b}} = V\setminus V_{\mathrm{g}}$ as the collection of \emph{bad} vertices. For each bad component $\mathsf C_{\mathrm b}$, i.e., a component in the induced subgraph of $\mathcal G$ on $V_{\mathrm b}$, let  $\partial_{\mathrm{e}}\mathsf C_{\mathrm{b}} = \{w\in V \setminus \mathsf C_{\mathrm b}: (w, u)\in \mathcal G \mbox{ for some } u\in \mathsf C_{\mathrm{b}}\}$ be the \textit{external boundary} of $\mathsf{C}_{\mathrm{b}}$ (we comment that by definition $\partial_{\mathrm{e}}\mathsf C_{\mathrm{b}} \subset V_{\mathrm{g}}$) and we let $\mathsf D(\mathsf{C}_{\mathrm{b}})$  be a graph on  $\mathsf C_{\mathrm b} \cup \partial_{\mathrm{e}} \mathsf C_{\mathrm b}$ which contains all edges within $\mathsf C_{\mathrm b}$ and all edges between $\mathsf C_{\mathrm b}$ and $\partial_{\mathrm{e}} \mathsf C_{\mathrm b}$. We would like to consider the empirical profile for $\mathsf D(\mathsf C_{\mathrm b})$ with $\mathsf C_{\mathrm b}$ ranging over all bad components. To this end, we let $\Psi$ be a mapping that maps each graph to its equivalence class where the equivalence is given by isomorphism that preserves good vertices. That is to say, we view $\Psi(\mathsf D(\mathsf C_{\mathrm b}))$ as a graph with vertices in $\partial_{\mathrm{e}} \mathsf C_{\mathrm b}$ labeled but with vertices in $\mathsf C_{\mathrm b}$ unlabeled. 
Define the profile $\mathcal D_{\mathrm b} = \{\Psi(\mathsf D(\mathsf C_{\mathrm b})): \mathsf C_{\mathrm b} \mbox{ is a bad component}\}$. Note that $\mathcal{D}_{\mathrm{b}}$ may be a set with multiplicity  since we may have multiple copies for a certain $\Psi(\mathsf D(\mathsf C_{  \mathrm b}))$ due to the fact that vertices in $\mathsf C_{ \mathrm b}$ are not labeled.
 
Our key intuition is to recover bad components from good vertices. To formalize this, for each $x\in V_{\mathrm g}$ let $U_x$ be the collection of vertices whose $\rho r$-neighborhoods are contained in $\mathsf N_{r-1}(x)$. We remark that in the definition above we require to be contained in $\mathsf{N}_{r-1}(x)$ instead of in $\mathsf{N}_{r}(x)$ for the reason that the former event is measurable with respect to $\mathsf N_r(x)$ but the latter is not. Write $U_{x, \mathrm g} = U_x \cap V_{\mathrm g}$ and $U_{x, \mathrm b} = U_x \cap V_{\mathrm b}$. 
We say that $\mathsf C_{x, \mathrm b}$ is a bad component in $\mathsf N_{r}(x)$ if
\begin{itemize}
\item $\mathsf C_{x, \mathrm b}$ is connected in the subgraph induced on $U_{x, \mathrm b}$;
\item $\partial_{\mathrm{e}} \mathsf C_{x, \mathrm b} \subset U_{x, \mathrm g}$ and $x\in \partial_{\mathrm{e}} \mathsf C_{x, \mathrm b} $ where
$$\partial_{\mathrm{e}} \mathsf C_{x, \mathrm b}  =  \{w\in \mathsf N_r(x) \setminus \mathsf C_{x, \mathrm b}: w \mbox{ is neighboring some } w' \in \mathsf C_{x, \mathrm b}\}\,.$$
\end{itemize}
Let $\mathfrak D_{x, \mathrm b}$ be the collection of $\mathsf D(\mathsf C_{x, \mathrm b})$ for all bad component $\mathsf C_{x, \mathrm b}$ in $\mathsf N_{r}(x)$, where (as above) $\mathsf D(\mathsf C_{x, \mathrm b})$ is a graph on $\mathsf C_{x, \mathrm b} \cup \partial_{\mathrm{e}} \mathsf C_{x, \mathrm b}$ which contains all edges within $\mathsf C_{x, \mathrm b}$ and all edges between $\mathsf C_{x, \mathrm b}$ and $\partial_{\mathrm{e}} \mathsf C_{x, \mathrm b}$.
We need to be careful about what we can recover exactly from a rooted neighborhood where all other vertices are not labeled. Since from $\mathsf N_r(v)$ we know the $\rho r$-neighborhoods for vertices in $\partial_{\mathrm e} \mathsf C_{x, \mathrm b}$ and since these vertices are good, we can assume that we know the labels for vertices in $\partial_{\mathrm{e}} \mathsf C_{x, \mathrm b}$; but we do not know labels for vertices in $\mathsf C_{x, \mathrm b}$. This echoes the definition of $\Psi$ from above. 
Similarly, $\mathfrak D_{x, \mathrm b}$ may be a set with multiplicity.
Furthermore, we define $\mathfrak D'_{x, \mathrm b}$ to be the collection of $\mathsf D(\mathsf C_{x, \mathrm b}) \in \mathfrak D_{x, \mathrm b}$ such that 
$
\mathsf D(\mathsf C_{x, \mathrm b}) \not\in \mathfrak D_{y, \mathrm b}$ for any $y \in \partial_{\mathrm{e}} \mathsf C_{x, \mathrm b}$ which is less than  $x$,
where we recall that we have fixed an arbitrary ordering on $V$.

We are now ready to recover our original graph by adding edges between $V_{\mathrm g}$ and then adding small components incident to $V_{\mathrm g}$ as follows:
\begin{itemize}
\item For any pair of good vertices, whether there is an edge can be determined by the $r$-neighborhood for either of them and we then add an edge if there is one.
\item For each good vertex $x$ and each $\mathsf D_{x, \mathrm b} (= \mathsf D(\mathsf C_{x, \mathrm b}))  \in \mathfrak D'_{x, \mathrm b}$, we add a copy of $\mathsf D_{x, \mathrm b}$ where bad vertices (i.e., those in $\mathsf C_{x, \mathrm b}$) in each such added copy are disjoint.
\end{itemize}
We denote by $\mathcal G' = (V', E')$ as the graph obtained from the preceding construction.

\noindent {\bf Running time analysis.} In our procedure, most operations are standard and can be performed in polynomial time except for the algorithm of testing isomorphism between two rooted neighborhoods. 
 So far there is no polynomial-time algorithm known to test isomorphism for general graphs, and the best result is a quasi-polynomial-time algorithm \cite{Babai16}. However, for $r$-neighborhoods under consideration in our problem, they are trees or tree-like graphs and efficient algorithms are known for isomorphism. More precisely, polynomial-time algorithms have been proposed to test isomorphism for graphs with bounded tree-width which in particular include graphs with bounded complexity; see \cite{Luks82, Bod90, ES17, LPPS17, GNSW20} (we also note that there is a classic linear-time algorithm to test isomorphism for rooted trees \cite{AHU74}). Thanks to Lemma~\ref{lemma-complexity-of-Nr} (below), with high probability all $r$-neighborhoods under consideration have bounded complexity and as a result isomorphism can be tested via polynomial-time algorithms (we may also simply stop the algorithm and declare failure if on the rare event the algorithm detects that some neighborhood has a complexity exceeding a prescribed bound, so that the algorithm stops in polynomial-time deterministically). Altogether, our recovery procedure has a polynomial running time. It is an interesting question  to design an algorithm that achieves the ``optimal'' running time.
 
 \medskip

The much more challenging task is to prove that the preceding procedure succeeds to recover the original graph with high probability. To this end, we need the following \emph{admissibility} condition for the  Erd\H{o}s-R\'enyi graph.
  
\begin{defn}\label{def-admissible}
We say $\mathcal G$ is $(r, \rho)$-admissible if
$$\mathcal D_{\mathrm b} = \cup_x \mathfrak D'_{x, \mathrm b}\,.$$
\end{defn}

\begin{lemma}\label{lem-admissibility-implies-success}
If $\mathcal G$ is $(r,\rho)$-admissible, then $\mathcal G'$ is isomorphic to $\mathcal G$.
\end{lemma}
\begin{remark}
Note that Lemma~\ref{lem-admissibility-implies-success} holds for all $r \geq 1$ and $\rho < 1$. The assumption \eqref{eq-rho-r-identifiable} made at the beginning of this section is for the purpose of verifying admissibility.
\end{remark}
\begin{proof}[Proof of Lemma~\ref{lem-admissibility-implies-success}]
In order to prove the lemma, we will define a vertex bijection $\varphi$ and we will prove that $\varphi$ is an isomorphism between $\mathcal G$ and $\mathcal G'$.
From our construction we see that $V_{\mathrm g} \subset V'$ and we define $\varphi$ to be the identical map on $V_{\mathrm g}$. It remains to define $\varphi$ on $V_{\mathrm b}$. Let $\mathcal D'$ be the isomorphic copies of $\cup_x \mathfrak D'_{x, \mathrm b}$ in $\mathcal G'$. By admissibility, there exists a bijection  $\Gamma: \mathcal D_{\mathrm b} \mapsto \mathcal D'$ such that $\mathsf D_{\mathrm b}$ is isomorphic to $\Gamma(\mathsf D_{\mathrm b})$
 for each $\mathsf D_{\mathrm b}  \in  \mathcal D_b$. In addition, we let $\varphi_{\mathsf D_{\mathrm b}}$ be an isomorphism; we remind the reader that $\varphi_{\mathsf D_{\mathrm b}}$ preserves good vertices. Since bad vertices in $\mathsf D_{\mathrm b}$'s are disjoint, we can then define
 $$\varphi(v) = \varphi_{\mathsf D_{\mathrm b}}(v) \mbox{ for } v \in V_{\mathrm b} \cap \mathsf D_{\mathrm b}\,.$$
Clearly $\varphi$ is a bijection.
 It remains to prove that $\varphi$ preserves edges. It is obvious that $\varphi$ preserves edges within $V_{\mathrm g}$, and thus it remains to check edges that are incident to at least one bad vertex. For each $\mathsf D_{\mathrm b} \in \mathcal{D}_{\mathrm b}$, we can write $\mathsf D_{\mathrm b} = \mathsf D(\mathsf C_{\mathrm b})$. In addition, for $\mathsf D'_{\mathsf b} = \Gamma(\mathsf D_{\mathrm b})$ we let $\mathsf C'_{\mathrm b}$ be the collection of vertices in  $\mathsf D'_{\mathsf b}$ but not in $V_{\mathrm g}$. From our construction, it is clear that $\mathsf C_{\mathrm b}$ is not neighboring any vertex outside the vertex set of $\mathsf D_{\mathrm b}$, and also $\mathsf C'_{\mathrm b}$ is not neighboring any vertex outside the vertex set of $\mathsf D'_{\mathrm b}$. For edges within $\mathsf D_{\mathrm b}$, $\varphi$ preserves them since the restriction of  $\varphi$ on $\mathsf D_{\mathrm b}$ is the same as $\varphi_{\mathsf D_{\mathrm b}}$ (which is an isomorphism between $\mathsf D_{\mathrm b}$ and $\mathsf D'_{\mathrm b}$). This completes the proof.
\end{proof}

In light of Lemma~\ref{lem-admissibility-implies-success}, the main remaining task is to prove that $\mathcal G$ is admissible with high probability. To this end, we present a sufficient condition for admissibility in this section and we verify this condition in Section~\ref{sec:ER-property}.

In light of our preprocessing, all remaining components are non-degenerate, i.e., it does not contain a degenerate vertex. In addition, the admissibility for each connected component implies admissibility for the whole graph. As a result, in this section we consider a connected non-degenerate graph $G = (V(G), E(G))$, and we define $V_{\mathrm g}(G)$ and $V_{\mathrm b}(G)$ as $V_{\mathrm g}$ and $V_{\mathrm b}$ above but with respect to graph $G$. 
For notation convenience, in many cases we drop the dependence on $G$ when there is no ambiguity. 

We next define \emph{cycle} and \emph{simple cycle}: we say a sequence of (not necessarily distinct) vertices is a cycle if each of the neighboring pairs (including the pair for the starting and ending vertices) is connected by an edge in $G$ and in addition all these edges are distinct; we say a cycle is a simple cycle if each vertex has degree $2$ in this cycle. We say an edge $e \in E(G)$ is a \emph{bridge} if it is not contained in any cycle in $G$.  
Let $E_{\mathrm{br}}(G)$ be the collection of bridges in $G$. Note that
if $\mathsf{T}$ is a connected component for the subgraph induced by $E_{\mathrm{br}}(G)$, then $\mathsf{T}$ must be a  tree; in this case we say $\mathsf T$ is a \emph{bridging-tree}.
Also, we denote by $\mathsf T_v$ the bridging-tree containing $v$ (this is well-defined since different bridging-trees are vertex disjoint).
In addition, we let $\partial_{\mathrm{i}} \mathsf T$ be the \emph{internal boundary} of $\mathsf{T}$,
consisting of vertices in $\mathsf T$ which are neighboring to some vertex outside of $\mathsf T$.  Furthermore, if $\mathsf{B}$ is a  connected component for the subgraph induced by $E(G)\setminus E_{\mathrm{br}}(G)$,
we say $\mathsf{B}$ is a \emph{block} of $G$. We note that $u$, $v$ are in the same block if and only if there is a cycle (not necessarily simple) containing $u$ and $v$.
Partly for the purpose of facilitating our understanding, we make some simple observations: (1) different blocks are vertex disjoint (so are different bridging-trees as we pointed out earlier); (2) a bridging-tree and a block share no common edge; (3) For a bridging-tree $\mathsf T$, we have that  $u \in \partial_{\mathrm{i}} \mathsf{T}$   if and only if  there is a (unique) block $\mathsf{B}$ such that $u = V(\mathsf{T}) \cap V(\mathsf{B})$ (this is implied by Lemma~\ref{blocktree} (i) below).

 We are now ready to define strongly-admissibility which guarantees admissibility (as shown in Proposition~\ref{prop-G-r-rho-L}).

\begin{defn}\label{def-G-r-rho-L}
Let $L, r \geq 1$ and $\rho \in (0,1)$.
We say $G$ is $(r, \rho, L)$-strongly-admissible if the following properties hold.
 \begin{enumerate}[(1)]
\item For every $v \in V_{\mathrm b}(G)$, if
$\mathsf{N}_{ \rho r}(v)$  has two $\rho r$-arms or  there is a cycle in $\mathsf{N}_{\rho r}(v)$ containing $v$,  then  there exists a unique simple cycle $\mathsf{O}$ in $\mathsf{N}_{\rho r}(v)$ containing $v$ and moreover $\mathrm{Length}(\mathsf{O}) \leq L$. In addition,
the connected component of $v$ in the subgraph induced by $E(G)\backslash E(\mathsf{O})$ is a bridging-tree in $G$ (namely $\mathsf{T}_{v}$) satisfying $H( \mathsf{T}_{v} ) \leq L$ and  $\partial_{\mathrm{i}} \mathsf{T}_{v}=\{v\}$.

 \item There are at most $\log r$ vertices in $G$ which are contained in cycles of lengths less than $ L$.
   \end{enumerate} 
  \end{defn}

  \begin{proposition}\label{prop-G-r-rho-L}
There exists $r_0 = r_0(\rho, L) \geq 1$ such that if $G$ is $(r, \rho, L)$-strongly-admissible for some $r\geq r_0$, then $G$ is $(r,\rho)$-admissible. 
 \end{proposition}

The rest of this section is devoted to the proof of Proposition~\ref{prop-G-r-rho-L}. To this end, we need the following two lemmas whose proofs are postponed until the end of this section. Recall that we have assumed $G$ is connected and non-degenerated.

\begin{lemma}\label{blocktree}
The following hold for a bridging-tree $\mathsf T$ of $G$:
 \begin{enumerate}[(i)]
\item For each $u \in \partial_{\mathrm{i}} \mathsf{T}$, let $G_{u}$ be the connected component of $u$ in the subgraph of $G$ induced by  edges in $E(G)\backslash E(\mathsf{T})$. Then
\begin{equation}\label{blocktree2}
 V(G)=  \bigcup_{u \in \partial_{\mathrm{i}} \mathsf{T}} V(G_{u}) \cup V(\mathsf{T}) \ , \  E(G)=  \bigcup_{u \in \partial_{\mathrm{i}} \mathsf{T}} E(G_{u}) \cup E(\mathsf{T})
\end{equation}
and $V(G_{u}) \cap V(\mathsf{T})=\{u\}$, $V(G_{u_{1}}) \cap V(G_{u_{2}}) = \emptyset$   for $u_{1}, u_{2} \in \partial_{\mathrm{i}} \mathsf{T}$ and $u_{1} \neq u_{2}$.
\item  For $v, w \in \mathsf{T}$, there exists a unique path (denoted by $[v,w]$) from $v$ to $w$ in $G$. For $y \in [v,w]$, let $V_{v;y}$ be  the collection of vertices  $u$ such that there is a path from $u$ to $v$ without visiting $y$. Then $ V_{v;y} \cap V_{w;y} = \emptyset$.
 \end{enumerate}
\end{lemma}

For $v \in V_{\mathrm{b}}$, let  $\mathsf{D}_{v}=\mathsf{D}(\mathsf{C}_{\mathrm b}(v))$ where $\mathsf{C}_{\mathrm b}(v)$ is the bad component containing $v$.  

 \begin{lemma}\label{ArmCharact}
There exists $r_0 = r_0(\rho, L) \geq 1$ such that   for all $r \geq r_0$ the following hold provided that 
$G$ is an
$(r,\rho,L)$-strongly-admissible graph:
\begin{enumerate}[(i)]
\item  There exists $x \in V_{\mathrm g}  $ such that $\mathsf{N}_{\rho r}(x)$ has two $\rho r$-arms.
  \item If $x \in V_{\mathrm g} \cap \mathsf{D}_v$  and $\mathsf{N}_{\rho r}(x)$ has two $\rho r$-arms,  then   $\mathsf{N}_{\rho r}(u) \subset \mathsf{N}_{r-1}(x)$ for all $u \in \mathsf{D}_v$. 
\end{enumerate} 
 \end{lemma}
 
We introduce yet another notation. Given two intersecting paths $P^{1}=(u_0,u_{1},\ldots,u_{m})$
and $P^{2}=(v_0,v_{1},\ldots,v_{\ell})$, let $k= \min\{j:u_{j} \in P^{2} \}$ and let $k'$ be such that $u_{k}=v_{k'}$. We define $g(P^{1}, P^{2}) $ to be the  path $ (u_0, u_{1}, \ldots,u_{k}=v_{k'}, v_{k'+1},\ldots, v_{\ell})$. We are now ready to prove Proposition~\ref{prop-G-r-rho-L}.

\begin{proof}[Proof of Proposition~\ref{prop-G-r-rho-L}]
It suffices to show that for all $v \in V_{\mathrm b}(G)$,
there exists $x \in  V_{\mathrm g} \cap \mathsf{D}_v$ such that
\begin{equation}\label{eq-x-characterizes-D(v)}
  \mathsf{N}_{\rho r} (u) \subset \mathsf{N}_{r-1}(x) \quad \text{ for all }  u \in \mathsf{D}_v\,.
\end{equation}
  The proof of \eqref{eq-x-characterizes-D(v)} proceeds as analysis by cases. To this end, we claim that for a bridging-tree $\mathsf{T}$, either $\partial_{\mathrm{i}} \mathsf{T} = \{v\}$ for some  $v \in V_{\mathrm b}$ or $\partial_{\mathrm{i}} \mathsf{T}  \subset V_{\mathrm g}$. To see this, note that if there is $v \in V_{\mathrm b}\cap \partial_{\mathrm{i}}\mathsf{T}$, then there is  a simple cycle containing $v$. If the length of this cycle is $> 2\rho r$, then $v$ has two $\rho r$-arms; if not, then this cycle is contained in $\mathsf{N}_{\rho r}(v)$. Thus, by Definition \ref{def-G-r-rho-L}  we have $\mathsf{T}=\mathsf{T}_{v}$ and $\partial_{\mathrm{i}} \mathsf{T}=\{v\}$, verifying the claim. In light of this claim, we only need to consider the following three cases (see Figures~\ref{Fig-Case-1-2} and \ref{Fig-d-Case-3} for illustrations).

\begin{figure}[htbp]
  \centering \includegraphics[scale=0.6]{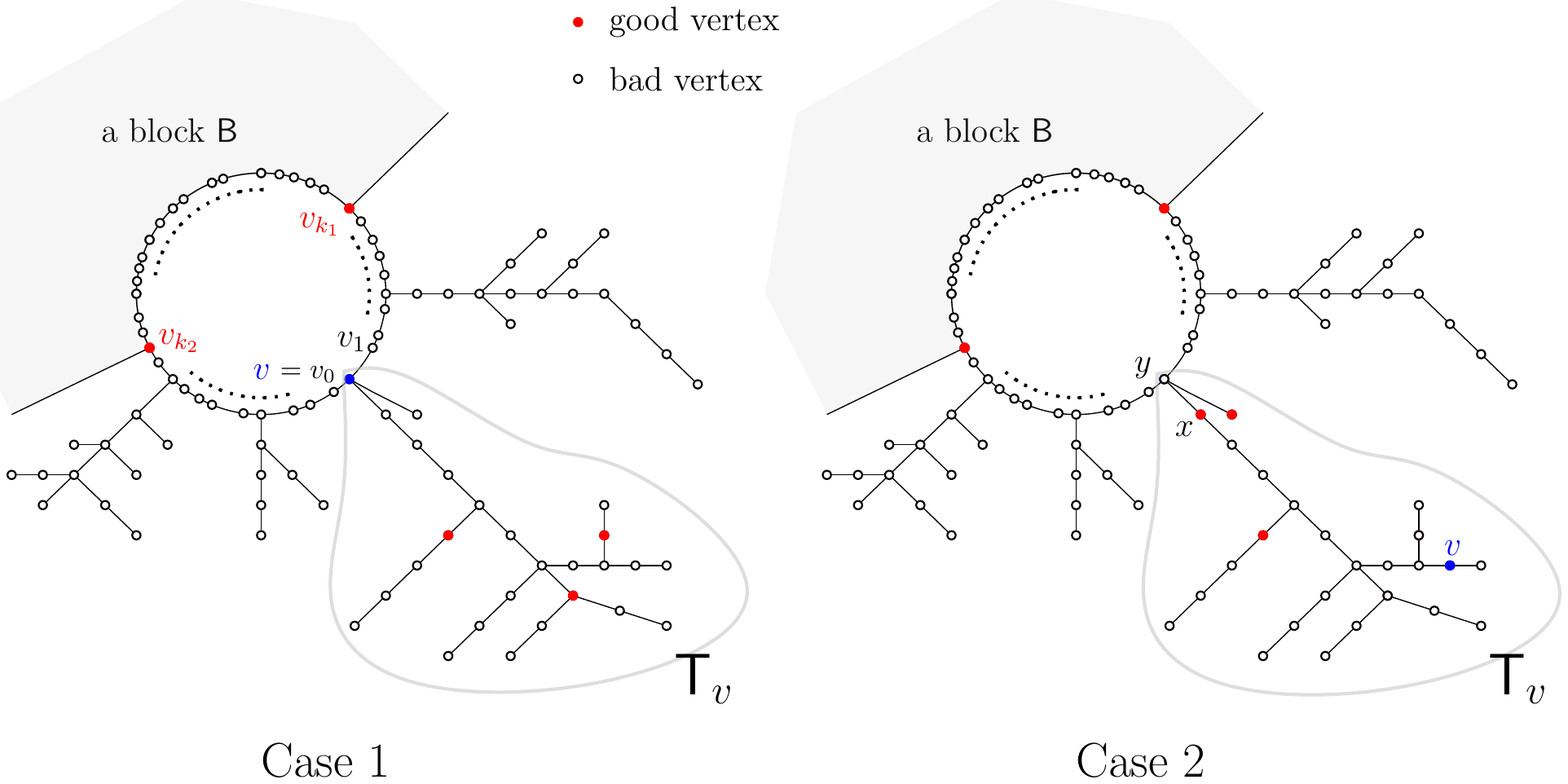}
  \caption{ Case 1 and Case 2 in Proposition \ref{prop-G-r-rho-L}}
\label{Fig-Case-1-2}
\end{figure}

\noindent {\bf Case 1:}  $v$ is contained in a cycle. By Definition~\ref{def-G-r-rho-L},  $v$ is contained in  a simple cycle  $\mathsf{O}=(v_{0}=v,v_{1},\ldots,v_{m} = v)$ with $m\leq L$. We claim that there must exist $u\in \mathsf O \cap V_{\mathrm g}$. Otherwise, we have $H(\mathsf T_w) \leq L$ for all $w\in \mathsf O$ by Definition~\ref{def-G-r-rho-L}. This would imply that the diameter of $G$ is at most $3L$ and thus $G$ is degenerate (assuming that $r$ is sufficiently large), arriving at a contradiction.

Now, let $k_{1}=\inf\{k \geq 1 : v_{k}\in V_{\mathrm g} \}$  and $k_{2}=\sup \{k \geq 1 : v_{k} \in V_{\mathrm g}\}$. Clearly $v_{k_{1}} \in V_{\mathrm{g}} \cap \mathsf{D}_{v}$.
By Definition \ref{def-G-r-rho-L}, if a path starting from $v_{j}$ (with $j<k_{1}$ or $j > k_{2}$) does not visit $v_{k_{1}}$ or $v_{k_{2}}$, then the path is contained in  $\cup_{k<k_{1}, k > k_{2}} V(\mathsf{T}_{v_{k}})$ where each tree satisfies $H(\mathsf{T}_{v_{k}}) \leq L$. Thus, we have $\mathsf{D}_v \subset\cup_{k<k_{1}, k > k_{2}} V(\mathsf{T}_{v_{k}})$. Since $m\leq L$, we see that $\mathsf D_v \subset \mathsf N_{2L}(v_{k_1})$ and  hence \eqref{eq-x-characterizes-D(v)} holds with $x = v_{k_1}$ (assuming that $r$ is sufficiently large).

\noindent {\bf Case 2:} $v$ is not contained in any cycle and $\partial_{\mathrm{i}} \mathsf T_v = \{y\}$ for some $y \in V_{\mathrm b}$. If $\mathsf{D}_y=\mathsf{D}_v$, then this reduces to Case 1. Thus, we may assume in addition that $\mathsf{D}_v \neq \mathsf{D}_y$. This implies that on the unique path from $v$ to $y$, there exists $x \in V_{\mathrm g}(\mathsf{D}_v) \cap V(\mathsf{T}_v)$,
  and hence $\mathsf{D}_v \subset \mathsf{T}_v=\mathsf{T}_{y}$. Since $H(\mathsf{T}_{y}) \leq L$ (as $y \in \partial_{\mathrm{i}} \mathsf T_v$ is contained in a cycle),  we have $\mathsf D_v \subset \mathsf N_{2L}(x)$ and hence \eqref{eq-x-characterizes-D(v)} holds.

\begin{figure}[htbp]
  \centering \includegraphics[scale=0.6]{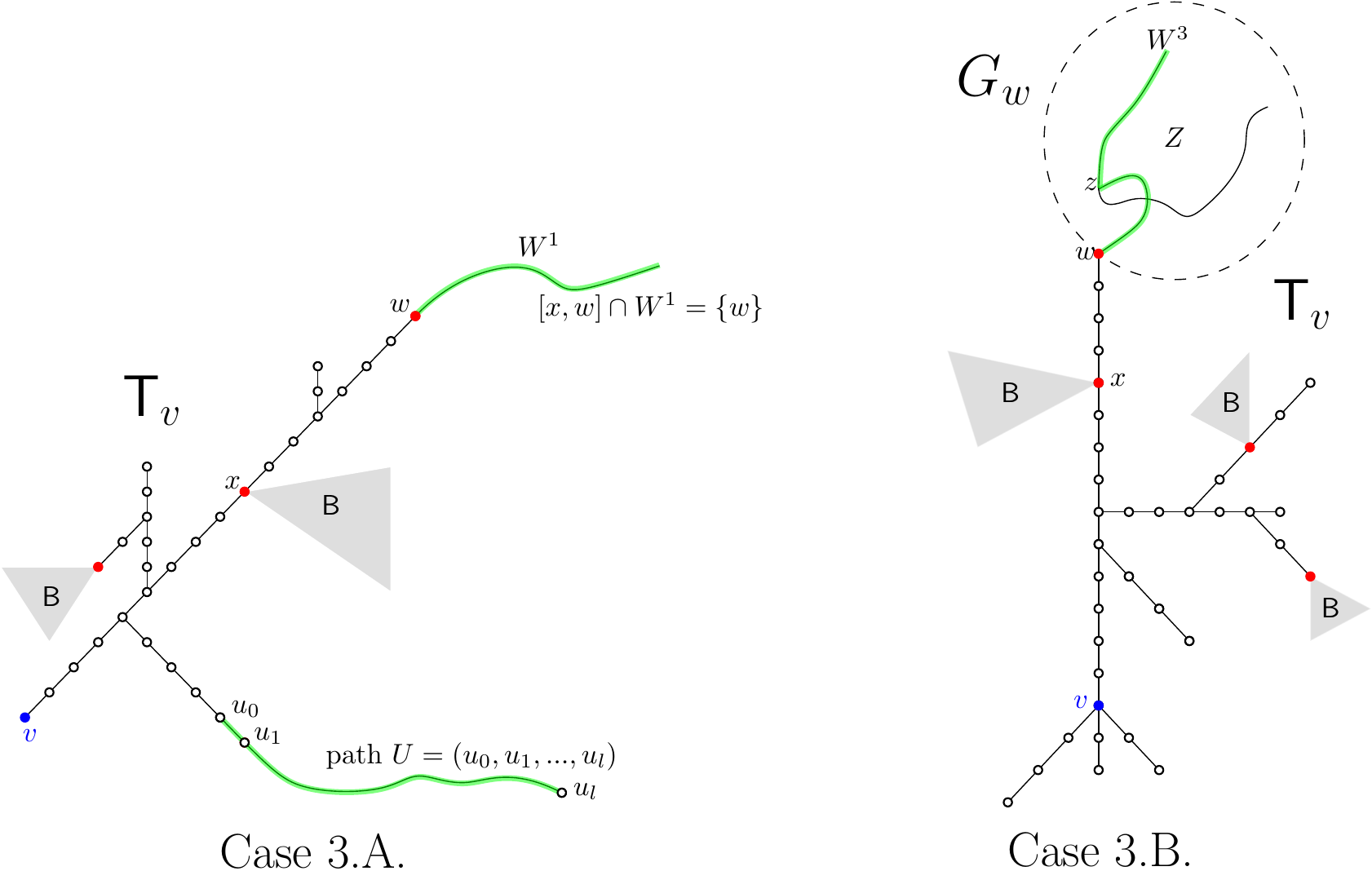}
  \caption{ Case 3  in Proposition \ref{prop-G-r-rho-L} }
\label{Fig-d-Case-3}
\end{figure}

\noindent {\bf Case 3:}  $v$ is not contained in any cycle and $\partial_{\mathrm{i}} \mathsf T_v \subset V_{\mathrm g}$. In this case, we have $\mathsf{D}_v \subset \mathsf{T}_v$, since  by \eqref{blocktree2} for $u \in \partial_{\mathrm{i}} \mathsf{T}_v$ one has that every path from $v$ to some $w \in G_{u}\backslash \{u\}$ visits $u$ (Recall that $G_u$ is defined in Statement (i) of Lemma~\ref{blocktree}). We further divide Case 3 into two subcases.

\noindent \underline{Case 3.A}: there exists  $w \in \mathsf{T}_v$ with two $\rho r$-arms. If $w\in V_{\mathrm b}$, then by Definition~\ref{def-G-r-rho-L} $w$ must be contained in a cycle and thus $w\in \partial_{\mathrm{i}} \mathsf T_v$, contradicting the assumption that $\partial_{\mathrm{i}} \mathsf T_v \subset V_{\mathrm g}$. As a result, $w$ must be good.

  Let $[v,w]$ be the unique  path  from $v$ to $w$ (using Lemma \ref{blocktree} (ii)), and let $x$ be the good vertex on $[v,w]$ which is closest to $w$. By the choice of ``closest'', we have $x \in \mathsf{D}_v$.  Let  $W^1,W^2$ be the two $\rho r$-arms of $w$.
By Lemma~\ref{blocktree} (ii), we have either $W^{1} \cap [w,x] =\{w\}$ or $W^{2} \cap [w,x] =\{w\}$ since otherwise $W^{1}$ and $W^{2}$ will not be edge-disjoint. Without loss of generality we assume that $W^{1} \cap [w,x] =\{w\}$.

If \eqref{eq-x-characterizes-D(v)} fails,  there exists a path $U=(u_0, u_{1}, \cdots, u_{\ell})$ with $\ell < \rho r$ such  that $u_0 \in \mathsf{D}_v\subset \mathsf{T}_v$ and $u_{\ell} \notin \mathsf{N}_{r-1}(x)$.  On the one hand, using the notation in Lemma \ref{blocktree} (ii) we have $u_0 \in V_{v;x}$.
Noting that  $u_{j} \neq x$ for all $j$, we then get   $u_{j} \in V_{v;x}$ for all $j$.
Let $P_{x \to u_0}$ be a path from $x$ to $u_0$ and let $P_{x \to u_{\ell}}=g(P_{x \to u_0}, U)$. Then the length of $P_{x \to u_{\ell}}$ satisfies $\mathrm{Length}(P_{x \to u_{\ell}})> \rho r$ since $u_{\ell} \notin \mathsf{N}_{r-1}(x)$. In addition, $P_{x \to u_{\ell}} \setminus \{x\} \subset V_{v;x}$.
On the other hand, the path $[x,w] \cup W^{1}$ has length $> \rho r$ and the vertices on this path (except $x$) are all in $V_{w;x}$. Thanks to Lemma \ref{blocktree}, $V_{v;x} \cap V_{w;x} = \emptyset$. Altogether, we get $x$ has two $\rho r$-arms, combining with $x\in V_{\mathrm{g}}$ yielding \eqref{eq-x-characterizes-D(v)} by Lemma \ref{ArmCharact}.

\noindent \underline{Case 3.B}:  there exists no vertex in $\mathsf{T}_v$ with two $\rho r$-arms. By Lemma~\ref{ArmCharact},
there exists $z \in G$
such that $\mathsf{N}_{\rho r}(z)$ has two $\rho r$-arms.
 Then   $ z \notin \mathsf{T}_v$.  Let $P_{v \to z}$ be a path from $v$ to $z$ and let $w$ be the last vertex on $P_{v \to z}$ such that $w \in \mathsf{T}_v$. Then clearly $w \in \partial_{\mathrm{i}} \mathsf{T}_v$ and the subpath $P_{w \to z} \subset G_{w}$.
Since $z$ has two $\rho r$-arms,  there is a path $Z$ with $2\rho r$ edges and with $z$ the middle point. Let $w'$ be the first point at which $P_{w \to z}$ intersects $Z$, and let $W^3$ be the path obtained by concatenating the subpath of $P_{w \to z}$  from $w$ to $w'$ and the longer subpath of $Z$ separated by $w'$.
Then $W^3  \subset G_{w}$  and has length $> \rho r$. 
 For our $v$, there is a unique  path $[v,w]$ from $v$ to $w$  and we let $x$ be the closest good vertex to $v$ on $[v, w]$. Then $x \in \mathsf{D}_{v}$ by definition.
 Using the same argument in the Case 3.A (replacing $W^{1}$ by $W^3$) we can show that \eqref{eq-x-characterizes-D(v)} holds.
  \end{proof}

   \begin{proof} [Proof of Lemma~\ref{blocktree}]
   We will employ proof by contradiction.

We first prove (i).   By definition, $E(G_{u}) \cap E(\mathsf{T})=\emptyset$ and $u\in V(G_{u}) \cap V(\mathsf{T}) $. If $w \in V(G_{u}) \cap V(\mathsf{T})$ for some $w\neq u$, then  there exists a path connecting $u$ and $w$ with edges in $E(\mathsf{T})$ and another path connecting $u$ and $w$ with edges in $E(G_{u})$. Thus, these two paths altogether form a cycle, contradicting the definition of bridging-tree. Therefore, $V(G_{u}) \cap V(\mathsf{T})= \{ u\}$.

 For $u_{1} \neq u_{2}$ in $ \partial_{\mathrm{i}} \mathsf{T}$, if $V(G_{u_{1}}) \cap V(G_{u_{2}})\neq  \emptyset$, then  $ V(G_{u_{1}}) = V(G_{u_{2}})$ (since $G_{u_1}$ and $G_{u_2}$ are connected components). Since we have shown that $V(G_{u_{i}}) \cap V(\mathsf{T})=\{u_{i}\}$ for $i=1,2$, this yields a contradiction. Therefore, $ V(G_{u_{1}}) \cap  V(G_{u_{2}})= \emptyset$ and consequently  $E(G_{u_{1}}) \cap E(G_{u_{2}}) = \emptyset $.

For $e \in E(G) \backslash E(\mathsf{T})$, let $z_0$ be an end-vertex of $e$ with $z_0 \not\in V(\mathsf T)$. Since $G$ is connected,  there exists a path $(z_{0}, z_{1}, \ldots, z_{m})$ so that $z_{m}$ is the only vertex in  $\mathsf{T}$.  This implies that $u = z_m \in \partial_{\mathrm{i}} \mathsf{T}$. In addition, by definition $e \in E(G_{u})$ and $z_0 \in V(G_{u} )$. Hence the decomposition \eqref{blocktree2} holds.

We next prove (ii). Let $[v,w]$ be the path from  $v$ to $w$ in the bridging-tree $\mathsf{T}$. If $(v=v_0, v_{1}, \ldots, v_{m}=w)$ is another path in  $G$. Let $k_{1}= \inf\{ j:v_{j} \notin [v,w] \}$ and  let  $k_{2}= \inf\{ j> k_{1} :v_{j} \in [v,w] \}$. Then we can see that the path $[v_{k_{2}}, v_{k_{1}}]$ in $\mathsf{T} $ and $(v_{k_{1}}, v_{k_{1}+1}, \ldots, v_{k_{2}} )$ form a simple cycle, contradicting with the definition of bridging-tree.
Finally, for $y \in [v,w]$, if $V_{v;y} \cap V_{w;y} \neq \emptyset$, then there exists a path from $v$ to $w$ without visiting $y$,  contradicting with the uniqueness of the path $[v,w]$.
\end{proof}

\begin{proof}[Proof of Lemma~\ref{ArmCharact}] 
  We first prove (i).
  Let $P$ be the longest path in $G$. Since $G$ is non-degenerate, we see that the length of $P$ is at least $2r - 2$ since otherwise $G \subset \mathsf N_{r-1}(u)$ (and thus $G$ is degenerate) where $u$ is the mid-point of $P$ (or one of the two mid-points of $P$ if $P$ has an even number of vertices). Therefore, there are at least $(r - \rho r - 1) \geq \log r$ (assuming that $r$ is sufficiently large) many vertices on $P$  which have two $\rho r$-arms. By Definition \ref{def-G-r-rho-L}, one of them must be good.

We now prove (ii).
We claim that, for every $y \notin \mathsf{N}_{r-1}(x)$ and every path $P_{y \rightarrow x}= (y_{0} = y, y_{1}, \ldots, y_{m} = x)$ from $y$ to $x$,  there exists  $j\in [m-\log r, m)$ such that $y_{j} \in V_{\mathrm g}$. Provided with this claim we now show that $\mathsf{N}_{\rho r+1}(v) \subset \mathsf{N}_{r-1}(x)$ if $x \in V_{g} \cap \mathsf{D}_{v}$. Denote by $\mathrm{dist}$ the graph distance on $G$. When $\mathrm{dist}(v,x)< \frac{(1- \rho) r}{4}$,  we have nothing to prove. Thus we may assume $\mathrm{dist}(v,x) \geq  \frac{(1- \rho) r}{4}$. Since $x \in \mathsf{D}_v$,
there exists a path $(x_0=x, x_{1}, \ldots, x_{m}=v)$ such that $x_{i} \in V_{\mathrm{b}}$ for $1\leq i\leq m$.
 If there exists $w \in   \mathsf{N}_{\rho r+1}(v) \setminus  \mathsf{N}_{r-1}(x)$, then there exists a geodesic $(v_0=v,v_{1},\ldots, v_{\ell}=w)$ from $v$ to $w$. Let $k=\inf\{i: x_{i} \in \{  v_{j}\} \}$ and let $k'$ be such that $x_{k}=v_{k'}$.
  Since $\mathrm{dist}(x_{k},w)< \mathrm{dist}(v,w)< \rho r+1$ and $\mathrm{dist}(x,x_{k}) \leq k$, we have $k \geq \mathrm{dist}(x,w)-\mathrm{dist}(x_{k},w) \geq \frac{(1-\rho)r}{4}$ by the triangle inequality.
  Note that $(x_0, \ldots, x_{k}=v_{k'}, v_{k'+1}, \ldots, v_{m})$ is a path connecting $x$ and $w \notin \mathsf{N}_{r-1}(x)$. Applying the claim to the reverse of this path,  we see that $\{x_{i}: 0<i \leq k\} \cap V_{\mathrm g} \neq \emptyset$ (assuming $r$ to be sufficiently large), arriving at a contradiction. Therefore, $  \mathsf{N}_{\rho r+1}(v) \subset \mathsf{N}_{r-1}(x)$. As a consequence, for all $u \in \mathsf{D}_v \cap V_{\mathrm b}$ we have $\mathsf{N}_{\rho r+1}(u) \subset \mathsf{N}_{r-1}(x)$  since $\mathsf{D}_v= \mathsf D_u$ (and thus we can apply the above reasoning with $v$ replaced by $u$). For  $u \in \mathsf{D}_v \cap V_{\mathrm g}$, there exists $y \in \mathsf{D}_v \cap V_{\mathrm b}$ neighboring to $u$, so $\mathsf{N}_{\rho r}(u) \subset \mathsf{N}_{\rho r + 1}(y) \subset \mathsf{N}_{r-1}(x)$.

It remains to prove the claim made at the beginning of the proof. Suppose $X^{1}$ and $X^{2}$ are two paths of length $\rho r$ and have the unique common vertex $x$.
If the path $P_{y \rightarrow x} \setminus \{x\}$ does not intersect with $X^{2}$, then  $u_{j}$ has two $\rho r $-arms $(y_0,\ldots,y_{j})$ and $(y_{j}, \ldots, y_{m}=x^{(2)}_0, \ldots, x^{(2)}_{\rho r})$ for all  $j \in [m-\log r, m)$. By Definition~\ref{def-G-r-rho-L}, one of these $y_j$'s must be good (since the number of such $y_j$'s is $\log r$), as desired. The case for $P_{u \rightarrow x} \setminus \{x\}$ does not intersect with $X^{1}$ can be treated similarly.

Finally, we consider the case when $P_{y \rightarrow x} \setminus \{x\}$ intersects with $X^{1}$ and $X^{2}$. Let $k_{j} = \sup\{i: y_{i} \in X^{j}\setminus \{ x  \} \} $ for $j=1,2$. Without loss of generality, assume that $k_{2}>k_{1}$. Let $\ell$ be such that  $u_{k_{1}} = x^{(1)}_{\ell}$.  Note that $\mathsf{O}=(x^{(1)}_0, \ldots, x^{(1)}_{\ell}=y_{k_{1}}, \ldots, y_{m}=x^{(1)}_{0} )$ is a simple cycle. Our proof proceeds by dividing into three cases depending on $\mathrm{Length}(\mathsf{O})$, i.e., the length of $\mathsf O$.

\noindent  Case 1: $\mathrm{Length}(\mathsf{O})> 2 \rho r$. Then  all the vertices on $\mathsf{O}$ (which include $u_{j}$ with $j \in [m-\log r, m) $) has two $\rho r $-arms. By Definition \ref{def-G-r-rho-L}, one of those $y_{j}$'s must be good.

\noindent  Case 2:  $  L<\mathrm{Length}(\mathsf{O}) < 2 \rho r $. Then by Definition \ref{def-G-r-rho-L} $y_{m-1}$ must be good (since $y_{m-1}$ is on a cycle of length larger than $L$).

\noindent Case 3: $\mathrm{Length}(\mathsf{O}) \leq L$. Then $k_{1}>m- \log r$ and $(y_{k_{1}}, \ldots, y_0) $ is a path of length $> \rho r$ which has disjoint edges with $\mathsf{O}$. By  Definition \ref{def-G-r-rho-L}, $y_{k_{1}}$ must be good.
   \end{proof}

 \section{Admissibility for Erd\H{o}s-R\'enyi graphs}\label{sec:ER-property}

In order to complete the proof of identifiability in Theorem~\ref{thm-main}, it suffices to prove the following result in light of Lemma~\ref{lem-admissibility-implies-success} and Proposition~\ref{prop-G-r-rho-L}. In what follows, we say a
graph is strongly-admissible, if each connected non-degenerate component of it is strongly-admissible.
\begin{proposition}\label{PropertyER}
Fix $\lambda, \epsilon_0>0$ and assume that $\rho, r$ satisfy  \eqref{eq-rho-r-identifiable}.
For any $\epsilon >0$, there exist  $N_{\epsilon} \in \mathbb{N}$  and $L=L_{\epsilon} \geq 1$ (both may depend on $\lambda$ and $\epsilon_0$) such that for all $n \geq N_{\epsilon}$,
$$\P(\mathcal{G}_{n,\frac{\lambda}{n}} \mbox{ is $(r, \rho, L)$-strongly-admissible}) \geq 1-\epsilon.$$
\end{proposition}
In the rest of the paper, unless otherwise specified, we assume that \eqref{eq-rho-r-identifiable} holds.

\subsection{Proof of Proposition~\ref{PropertyER}}

In this subsection, we prove Proposition~\ref{PropertyER} with postponing the proof for Lemma~\ref{lemma-reducedBFS} (below) to later subsections. To this end, we will employ breadth-first-search (BFS) process simultaneously from a pair of vertices $u, v\in \mathcal G$ (which we refer to as \emph{reduced BFS}). In order to approximate their $r$-neighborhoods by  independent  PGW$(\lambda)$ trees, we will need some kind of ``cut off'' and ``graft'' operations as we describe in what follows.

We now describe our reduced BSF (with respect to $u, v\in \mathcal G$) which is a modification of the standard BFS. As a comment on notation below, we will denote by $A_{t}$ for \emph{active} vertices, i.e., we are going to explore their neighbors;
we denote $R_{t}$ for \emph{removed} vertices, i.e., we will not  explore  their neighbors; we denote $U_{t}$ for \emph{unexplored} vertices, i.e., these vertices have not be explored as neighbors of some active vertices.
Initially we set $  R_0=\emptyset$,
    $A_0(v)=\{v\}$,  $A_0(u)=\{u\}$,  $A_0 = A_{0}(u) \cup A_{0}(v)$  and $U_0= V \backslash A_0$.
For $t\geq 0$, as long as $A_{t}(u)\cup A_{t}(v) \neq \emptyset$, we make the following recursive definition (which corresponds to a ``search process'' from $A_t(u) \cup A_t(v)$) where the set $R_t$ corresponds to the aforementioned ``cut off'' operation:
    \begin{equation}\label{eq-def-Reduce-BFS}
\begin{split}
    R_{t+1} &= \{ w \in U_{t} : \exists x \in A_{t}(u), y \in A_{t}(v) , \mathcal{G}_{xw} =\mathcal{G}_{yw} =1  \}\,; \\
    A_{t+1}(u)&= \{ w \in U_{t} :  \exists x \in A_{t}(u)  , \mathcal{G}_{xw}   =1  \} \backslash R_{t+1} \, ; \\
    A_{t+1}(v)&= \{ w \in U_{t} :  \exists y \in A_{t}(v)  , \mathcal{G}_{yw}   =1  \} \backslash R_{t+1} \, ; \\
    U_{t+1} &= U_{t} \backslash (A_{t+1}(u) \cup A_{t+1}(v) \cup R_{t+1})\,.
\end{split}
    \end{equation}
 (We note that  $A_{t}(u), A_{t}(v), R_{t}$ for $ t \geq 0$ are pairwise disjoint.)

 We next inductively construct two rooted trees $\mathcal{T}_{\mathrm{cut}}(u)$ and $\mathcal{T}_{\mathrm{cut}}(v)$ with  vertex sets $\cup_{t}A_{t}(u)$ and $\cup_{t} A_t(v)$ respectively.  Recall that we have a pre-fixed ordering on $V$. For $t\geq 0$, we assume the first $t$-levels of $\mathcal{T}_{\mathrm{cut}}(v)$ and $\mathcal{T}_{\mathrm{cut}}(u)$ have been defined (and their $t$-th levels  are $A_{t}(v)$ and $A_t(u)$, respectively).
We arrange the vertices in $A_{t}(v)$ as $v_{t,1}, \ldots, v_{t,|A_{t}(v)|}$ according to our pre-fixed order.
 Then for each $j$,  we  add the edges
 \begin{equation*}
  \left\{ ( v_{t,j},  y) : \mathcal{G}_{v_{t,j} y}=1, y \in U_{t} \backslash R_{t+1} \mbox{ and } \mathcal{G}_{v_{t,i}y}=0 \mbox{ for } 1 \leq i < j  \right\}
 \end{equation*}
 to $E(\mathcal{T}_{\mathrm{cut}}(v))$. Note that the $(t+1)$-th level of $\mathcal{T}_{\mathrm{cut}}(v)$ is exactly $A_{t+1}(v)$.   This gives the tree $\mathcal{T}_{\mathrm{cut}}(v)$ and   we define   $\mathcal{T}_{\mathrm{cut}}(u)$ similarly.

Furthermore, we define the auxiliary tree $\mathcal{T}_{\mathrm{aux}}(v)$  as an enlargement of $\mathcal{T}_{\mathrm{cut}}(v)$  as follows.
For each $y \in \mathcal{T}_{\mathrm{cut}}(v) \cap A_{t}(u)$, if there is $w \in R_{t+1}$ such that $\mathcal{G}_{yw}=1$, we add an independent $\mathrm{Bin}(n, \lambda/n)$-Galton-Watson tree rooted at a \emph{copy} of $w$ (where all other vertices are labeled as $\infty$) to  $\mathcal{T}_{\mathrm{cut}}(v)$ by connecting this copy of $w$ to $y$ (as the child of $y$). This corresponds to the aforementioned ``graft" operation.
(Note that it may be slightly more natural to add a PGW($\lambda$) tree, and we chose to add a $\mathrm{Bin}(n, \lambda/n)$-GW tree just for the slight convenience of exposition in the proof of Lemma~\ref{lemma-aux-tree-vs-GW-tree}.)
Then we get a rooted   tree $\mathcal{T}_{\mathrm{aux}}(v)$. Similarly, we can define $ \mathcal{T}_{\mathrm{aux}}(u)$.
Indeed, $\mathcal{T}_{\mathrm{cut}}(v)$ can be obtained as a subtree of $\mathcal{T}_{\mathrm{aux}}(v)$ by deleting all the vertices whose labels also appear in $\mathcal{T}_{\mathrm{aux}}(u)$ (and similarly for $\mathcal{T}_{\mathrm{cut}}(u)$).


 \begin{figure}[htbp]
      \centering
      \includegraphics[scale=0.7]{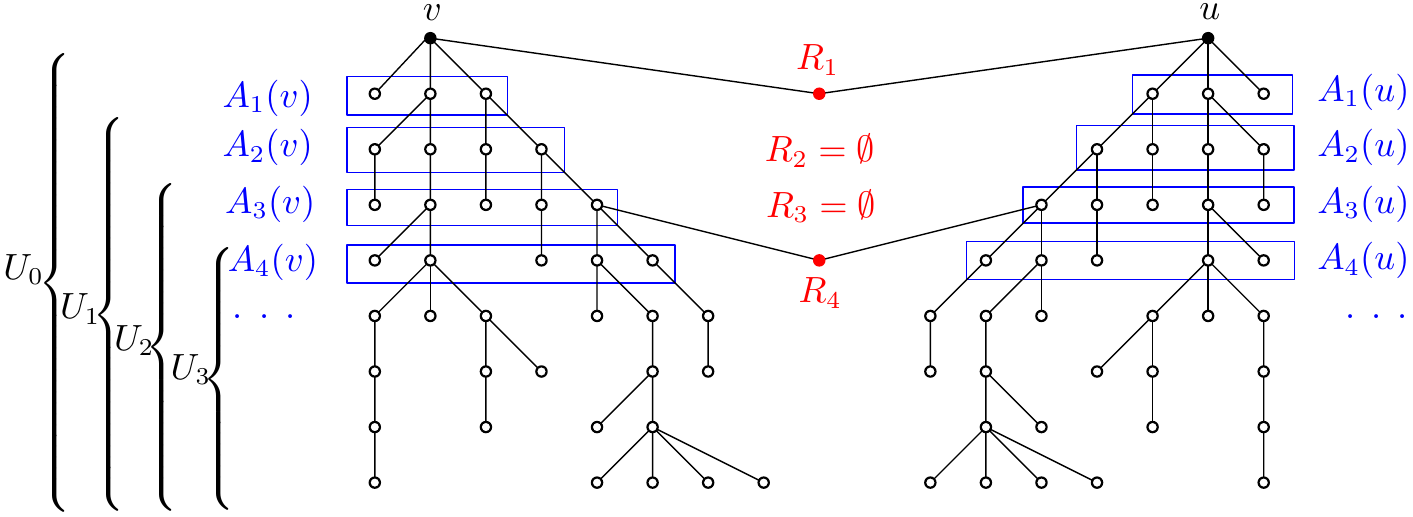}
     \caption{Example of reduced BFS} \label{p5}
 \end{figure}

Suppose that $u,v \in \mathcal{G}$ and  $\mathsf{N}_{r+1}(u) \sim \mathsf{N}_{r+1}(v) $. Let $\phi$ be an isomorphism from $\mathsf{N}_{r+1}(u)$ to $\mathsf{N}_{r+1}(v) $ such that $\phi(u)=v$. We claim  that
\begin{equation}\label{eq-Av-onto-Au}
  \phi(A_{t}(u) )=   A_{t}(v) \mbox{ and } \ \phi(R_{t}) = R_{t}  \, \text{ for all } t \leq r \,.
\end{equation}
Clearly \eqref{eq-Av-onto-Au} holds for $t=0$. Assume that \eqref{eq-Av-onto-Au} holds for $0,\ldots, t-1$, and $t\leq r$. Then for any $y \in A_{t}(u)$,  by \eqref{eq-def-Reduce-BFS} and by our choice of $\phi$ we have  $\phi(y) \in A_{t}(v) \cup R_{t}$. 
If $\phi(y) \in R_{t}$, then there is   $z \in A_{t-1}(u)$ such that $\mathcal{G}_{z,\phi(y)}=1$. Since $\phi(y) \in R_{t}$ and $t\leq r$, we have $z \in \mathsf{N}_{r+1}(v)$, and hence  $(z, \phi(y)) \in E(\mathsf{N}_{r+1}(v))$. Hence  $(\phi^{-1}(z), y) \in E(\mathsf{N}_{r+1}(u)) $. By our induction hypothesis  $\phi^{-1}(z) \in A_{t-1}(v)$, so $y \in R_{t}$, which contradicts with $y\in A_t(u)$. Thus $\phi(y) \in A_{t}(v)$. This implies that $\phi(A_{t}(u))=A_{t}(v)$, and as a result we also have $\phi(R_{t})=R_{t}$.

We  need some  more quantities to describe the structure of the $r$-neighborhoods of $u,v $.  
Let $A_{\leq r}(u) =\cup_{t=0}^{r}A_{t}(u)$,  $R_{\leq r} =\cup_{t=0}^{r}R_{t}$, and  $A_{\leq r} = A_{\leq r}(u)\cup A_{ \leq r }(v)$.

\begin{itemize}
\item Let $\Xi_{1} = \sum_{t=0}^{r}\sum_{x \in A_{t}(u), y \in A_{t}(v)} \mathcal{G}_{xy}$. If $x \in A_{t}(u), y \in A_{t'}(v)$ and $\mathcal{G}_{xy}=1$, by the definition of reduced BFS (see \eqref{eq-def-Reduce-BFS})  we have  $t=t'$.  Hence $\Xi_{1}  = \sum_{x \in A_{\leq r}(u), y \in A_{\leq r}(v)} \mathcal{G}_{xy} $.
\item  Let $\Xi_{2} = \sum_{t=0}^{r-1}\sum_{w \in U_{t}} \mathcal{G}_{w, A_{t}(u)} \mathcal{G}_{w, A_{t}(v)}$, where $\mathcal{G}_{w,A}= \sum_{y \in A} \mathcal{G}_{wy}$ for a subset $A$.
We can see that  $ |R_{\leq r}|=
\sum_{t=0}^{r-1} \sum_{w \in U_{t}}  1_{\{\mathcal{G}_{w,A_{t}(u) } \geq 1 \}}   1_{\{\mathcal{G}_{w,A_{t}(v) } \geq 1 \}}    $, and hence  $|R_{\leq r}|  \leq \Xi_{2}  $.

\item Let $\Lambda_{1}(u)=\mathrm{Comp}(\mathcal{G}[A_{\leq r}(u)]) $, where $\mathcal{G}[A_{\leq r}(u)]$ is the  subgraph on $\mathcal{G}$ induced by $ A_{\leq r}(u)$ and (we recall that) $\mathrm{Comp}(G)$ is the \emph{complexity} for a graph $G$ (that is, $\mathrm{Comp}(G)$ is the minimal number of edges that one has to remove from $G$ so that no cycle remains). 
 Similarly let $\Lambda_{1}(v)= \mathrm{Comp}(\mathcal{G}[A_{\leq r}(v)]) $.  Let $\Lambda_{1} = \Lambda_{1}(u)+ \Lambda_{1}(v)$. By \eqref{eq-Av-onto-Au}, if $\mathsf{N}_{r+1}(u) \sim \mathsf{N}_{r+1}(v) $, we have $\Lambda_{1}(u) =\Lambda_{1}(v) $.

\item Write $A_{[t,r]}(u)=\cup_{s=t}^{r}A_{s}(u)$ for $t\leq r$ and define
\begin{equation*}
  \Lambda_{2}(u)=\sum_{t=0}^{r} \sum_{w \in R_{t}, x \in A_{ [t,r]}(u)} \mathcal{G}_{w,x}  +
\sum_{x \in A_{\leq r}(u)}\,\sum_{w \in R_{\leq r}} \sum_{y \in \mathsf{N}_{r}(w) \cap \mathsf{N}_{r}(u) \backslash (A_{\leq r}\cup R_{\leq r}) }   \mathcal{G}_{xy}\,.
\end{equation*}
Similarly we define $\Lambda_{2}(v)$.  We can see that when $\Xi_{2}=|R_{\leq r}|$ and $\Lambda_{1}(u)=\Lambda_{2}(u)=0$, for each $w \in R_{\leq r}$ there exists a unique path in $\mathcal{G}[A_{\leq}(u)]$ from $u$ to $w$. 
Let $\Lambda_{2}= \Lambda_{2}(u)+ \Lambda_{2}(v)$. By \eqref{eq-Av-onto-Au}, if $\mathsf{N}_{r+1}(u) \sim \mathsf{N}_{r+1}(v) $, then $\Lambda_{2}(u)=\Lambda_{2}(v)$.

  \item  Let $\Lambda_{3}$ be the indicator function of the event that there are $w_{1} \neq  w_{2}$ in $R_{\leq r}$ which are connected by a path in $(\mathsf N_r(u)\cup \mathsf N_r(v)) \setminus A_{\leq r}$.  
\end{itemize}
Finally, let $\Xi(u, v) = \Xi =\Xi_{1}+ \Xi_{2}$ and let $\Lambda(u, v) = \Lambda=  \Lambda_{1}+\Lambda_{2}+\Lambda_{3} $.

\begin{lemma}\label{lemma-reducedBFS}
Assume \eqref{eq-rho-r-identifiable}. For any $\epsilon>0$, there exist $N_{\epsilon} \in \mathbb{N}$ and $ L_\epsilon \geq 1$ (which may depend on $\lambda$ and $\epsilon_0$) such that  for every $n \geq N_{\epsilon}$ the following holds
with probability at least $1-\epsilon$.  For any two vertices  $u$, $v$ in  $\mathcal{G}_{n,\frac{\lambda}{n}}$   such that $\mathsf{N}_{r+1}(u) \sim \mathsf{N}_{r+1}(v) $ and $\mathsf{N}_{r}(v)$ survives (recall that this means $\mathsf N_r(v) \neq \mathsf N_{r-1}(v)$), we have
\begin{enumerate}[(i)]
\item $\Xi   \leq 2$. As a consequence,  $\Xi_{2}=|R_{\leq r}|$.
\item If $\Xi =0 $, then  $\Lambda=0$, and $u$ does not have two $r$-arms.
\item If $\Xi=1$, then  $\Lambda=0$, and
$H(\mathcal{T}_{\mathrm{cut}}(v) )    <  \rho' r$, where $\rho'= \rho'(\epsilon_0)=\frac{1+ \epsilon_0/2}{1+\epsilon_0 }$.
  \item If $\Xi = 2 $, then  $\Lambda=0$, and
   $H(\mathcal{T}_{\mathrm{cut}}(v) )    \leq  L_{\epsilon}$.  
\end{enumerate}
\end{lemma}

\begin{proof} [Proof of Proposition~\ref{PropertyER}]
Assuming that   $\mathcal{G}$ satisfies the properties in the statement of Lemma \ref{lemma-reducedBFS}, we claim that   for every non-degenerate $v $ with non-unique $(r+1)$-neighborhood, if $v$ has two $r$-arms or there is a cycle in $\mathsf{N}_{r+1}(v)$ containing $v$, then there exists a unique simple cycle $\mathsf{O}$ in $\mathsf{N}_{r+1}(v)$ containing $v$ and moreover $\mathrm{Length}(\mathsf{O}) \leq 4L_\epsilon$. The connected component of $v$ in the subgraph induced by $E(\mathcal{G})\backslash E(\mathsf{O})$ is a bridging-tree in $\mathcal{G}$ (which is $\mathsf{T}_{v}$) satisfying $H( \mathsf{T}_{v} ) \leq L_\epsilon$ and  $\partial_{\mathrm{i}} \mathsf{T}_{v}=\{v\}$.
Since $ r = \frac{(1+\epsilon_0) \log n}{\log \alpha_\lambda^{-1}}$
and we choose $\epsilon_0$ arbitrarily,  the  claim above also holds if we replace $r$ by $\rho r -1$ (the replacement also occurs in Lemma~\ref{lemma-reducedBFS}), which is (1) in Definition \ref{def-G-r-rho-L} (with $L$ replaced by $4 L_\epsilon$).   In addition, it is well known that  the number of  cycles of length $\ell$ in $\mathcal{G}$ converges to a Poisson random variable (see e.g., \cite[Corollary 4.9]{Bollobas01}). Thus, (2) in Definition \ref{def-G-r-rho-L} holds.
Therefore, it remains to prove the above claim.
 
Assume that $\mathsf{N}_{r +1 }(u) \sim \mathsf{N}_{r+1 }(v) $ for some $u \neq v$ in $\mathcal{G}$.  Then we do the reduced BFS  for $u,v$. By Lemma \ref{lemma-reducedBFS},  we have $\Xi =  \Xi(u, v)\leq 2$. We next show that $\Xi$ is not 0 or 1.
\begin{itemize}
\item If $\Xi(u, v)=0$, then by Lemma \ref{lemma-reducedBFS} $\Lambda(u, v)=0$ and $v$ does not have two $r$-arms. Also, $\Xi=\Lambda=0$ implies that $v$ is not on any cycle, arriving at a contradiction.
\item If $\Xi(u,v)=1$, by Lemma \ref{lemma-reducedBFS} we have   $\Lambda(u, v)=0$  and $H(\mathcal{T}_{\mathrm{cut}}(v)) \leq  \rho' r$. When $\Xi=\Xi_{2}=1$, let $P_{v \to w}$ be the unique path from $v$ to the 
 vertex $w \in R_{\leq r}$;  when $\Xi=\Xi_{1}=1$, let $P_{v \to w}$ be the unique path to the  vertex  $w \in A_{\leq r}(v)$ such that there exists $y \in A_{\leq r}(u)$ with $\mathcal{G}_{wy}=1$.
In both scenarios,  every path starting from  $v$ and not contained in $\mathcal{T}_{\mathrm{cut}}(v)$ must contain $P_{v \to w}$.
  If there is a simple cycle containing $v$, then this cycle also contains a vertex $x \notin \mathcal{T}_{\mathrm{cut}}(v)$ as $\Lambda_{1}=0$. Thus, there are two edge-disjoint paths from $v$ to $x$ both of which contain $P_{v \to w}$, arriving at a contradiction. Therefore, $v$ is not contained in any simple cycle.  In addition,  since $\Lambda=0$ and $H(\mathcal{T}_{\mathrm{cut}}(v)) \leq \rho' r$, every path from $v$ with length $r$ must contain a vertex not in $\mathcal{T}_{\mathrm{cut}}(v)$. By the same argument, $v$ does not have two $r$-arms either, arriving at a contradiction.
\end{itemize}

Therefore it must be $\Xi=2$ and $\Lambda=0$  by Lemma \ref{lemma-reducedBFS}. We may assume that $\Xi=\Xi_{2}=|R_{\leq r}| = 2$ and write $R_{\leq r}=\{w_{1},w_{2}\}$ (the other case can be  proved in the same manner).
For $i=1,2$, let $P_{v \to w_{i}}$ be the  path from $v$ to $w_{i}$  consisting of vertices in $\mathcal{T}_{\mathrm{cut}}(v)\cup \{w_{i}\}$, which is  unique as $\Lambda=0$. We similarly define $P_{u \to w_{i}}$ (note that here for instance $P_{u\to w_1}$ is the reverse of $P_{w_1 \to u}$).
Then $\mathsf{O}=P_{v \to w_{1}}\cup P_{ w_{1} \to u} \cup P_{u \to w_{2}} \cup P_{w_{2} \to v}$ is a cycle  in $\mathsf{N}_{r}(v)$ containing $v$.
Since   $H(\mathcal{T}_{\mathrm{cut}}(v)) \leq   L_{\epsilon}$,
we have $\mathrm{dist}(w_{i},v) \leq L_{\epsilon}$ for $i=1,2$. Similarly we have $\mathrm{dist}(w_{i}, u) \leq L_{\epsilon}$ for $i=1,2$. Therefore,  $\mathrm{Length}(\mathsf{O}) \leq  4L_{\epsilon}$.

We now show that $\mathsf{O}$ is simple. 
It suffices to show that $P_{v \to w_{1}} \cap  P_{v \to w_{2}} = \{v\}$ and $P_{u \to w_{1}} \cap  P_{u \to w_{2}} = \{u\}$.  We will prove the statement for $v$ (and that for $u$ can be proved similarly), which is divided into two cases.
\begin{itemize}
\item If $v$ has two $r$-arms, each of the two $r$-arms must visit  a vertex not in $\mathcal{T}_{\mathrm{cut}}(v)$ since $H(\mathcal{T}_{\mathrm{cut}}(v) ) \leq L_{\epsilon}$. Since $\Lambda=0$, these two $r$-arms must contain   $P_{v \to w_{1}}$, $P_{v \to w_{2}}$, respectively.   So $P_{v \to w_{1}} \cap  P_{v \to w_{2}} = \{v\}$.
\item If $v$ is contained by a simple cycle in $\mathsf{N}_{r+1}(v)$,  from $\Lambda=0$ we see there exists a vertex  $y \in A_{\leq }(u)$  on this simple  cycle.  Thus from $v$ to $y$ there are two  paths  only intersecting at $v$ and $y$.  Since $\Lambda=0$, these two  paths must contain   $P_{v \to w_{1}}$, $P_{v \to w_{2}}$, respectively.
So $P_{v \to w_{1}} \cap  P_{v \to w_{2}} = \{v\}$.
\end{itemize}
Combining with $\Lambda=0$ and $\Xi = 2$, we can then further deduce that $ \mathsf{O}$ is the unique simple cycle containing $v$  in $\mathsf{N}_{r+1}(v)$.

In addition, the connected component of  $v$ in $E(\mathcal{G})\backslash E(\mathsf{O})$  is $  \mathcal{T}_{\mathrm{cut}}(v)\backslash (P_{v \to w_{1}}\cup P_{v \to w_{2}} )$. Thus this component is a tree and has height at most  $L_{\epsilon}$. Since $\Lambda=0$,  every edge in this tree is not contained by any cycle in $\mathcal{G}$, yielding that this tree is a bridging-tree in $\mathcal{G}$ (i.e, it is $\mathsf{T}_{v}$). 
This completes the proof of (1) in Definition \ref{def-G-r-rho-L}.
 \end{proof}

It remains to prove Lemma~\ref{lemma-reducedBFS}. To this end, we prove additional properties for PGW trees in Section~\ref{sec:additional-GW-trees} and then provide the proof for Lemma~\ref{lemma-reducedBFS} in Section~\ref{sec:lem-BSF}.

\subsection{Additional properties for Galton-Watson trees}\label{sec:additional-GW-trees}

We  need to control the volume growth for a Galton-Watson tree as incorporated in Lemma~\ref{moments}; this is fairly standard and we include a proof only for completeness. The new ingredient of significance to us is the conditioning on isomorphism as in Lemma~\ref{cdtmoment}.

\begin{lemma}\label{moments}
Fix $\lambda \geq 1$. Let $(Z_{\ell})_{\ell \geq 0}$ be the number of vertices in the $\ell$-th level of a PGW($\lambda$)-tree. Then for every $m\geq 1$, there is a constant $C_{m,\lambda} > 0$ depending only on  $\lambda$ and $m$ (we denote by $C_m = C_{m, 1}$ for short) such  that
\begin{enumerate}[(i)]
   \item when $\lambda = 1$, $ \E[ Z_{\ell}^{m} ] \leq  C_{m} \ell^{m-1} $ for all $\ell \geq 1$;
   \item when $\lambda  > 1$, $ \E[ Z_{\ell}^{m} ] \leq  C_{m,\lambda} \lambda^{\ell m} $ for all $\ell \geq 1$.
\end{enumerate}
As a consequence, for  every $\theta > \log (\lambda ) \geq 0$,
\begin{equation}\label{eq-consequence}
  \lim_{r \to \infty} \frac{1}{r} \log \P \big( \mbox{$\sum_{\ell=0 }^{r}$} Z_{\ell} > e^{\theta r} \big) = - \infty \,.
\end{equation}
\end{lemma}

\begin{remark}\label{rmk-moments}
The same proof for Lemma~\ref{moments} below easily gives the same result for a Galton-Watson tree with offspring distribution as $\mathrm{Bin}(n, \lambda/n)$. As a result, in what follows we also apply Lemma~\ref{moments} in this case.
\end{remark}

\begin{proof}[Proof of Lemma~\ref{moments}]
We prove  (i) and (ii) by induction on $m$. The base case for $m = 1$ holds obviously.
Assume they hold for $1, \cdots, m-1$.
Since the conditional law of $Z_\ell$ given $Z_{\ell-1}$ is Poisson with mean $\lambda Z_{\ell-1}$,  we have that
\begin{equation*}
\E[ Z_{\ell}(Z_{\ell}-1) \ldots  (Z_{\ell}-m+1) |  Z_{\ell-1}] = \lambda^{m} Z_{\ell-1}^{m} \,.
\end{equation*}
Let  $C'_m= \sum_{j=0}^{m}\binom{m}{j} (m+1)^{m}$  we have
\begin{equation*}
  \E[ Z_{\ell}^{m}  ] \leq \lambda^{m} \E[ Z_{\ell-1}^{m}] + C'_m  \sum_{1 \leq j \leq m-1}\E[ Z_{\ell}^{j}]\,.
\end{equation*}
Then by our induction hypothesis we get that for all $\ell \geq 1$
\begin{align*}
  \E[ Z_{\ell}^{m}  ] & \leq   \E[ Z_{\ell-1}^{m}] +  (m C'_m  \max_{  j \leq m-1} C_{j } )  \ell^{m-2} \mbox{ for } \lambda = 1\,,\\
  \E [  (\tfrac{Z_{\ell}}{\lambda^{\ell}}) ^{m}  ] &\leq   \E[ (\tfrac{Z_{\ell-1}}{\lambda^{\ell-1}}) ^{m}] + (  m C'_m \max_{  j \leq m-1} C_{j,\lambda } ) \tfrac{1}{\lambda ^{\ell  }} \mbox{ for } \lambda >1\,.
\end{align*}
This completes the proof of (i) and (ii) by induction and by choosing  $C_m$ and $C_{m, \lambda}$ appropriately.

We next prove \eqref{eq-consequence}. For every $m \geq 1$, writing $\beta_\ell = \frac{\ell^{-2}}{\sum_{i\geq 1} i^{-2}}$ we have that
  \begin{equation*}
     \P \big( \mbox{$\sum_{\ell=0 }^{r} Z_{\ell}$} > e^{\theta r} \big)  \leq   \sum_{\ell=0 }^{r} \P( Z_{\ell} > e^{\theta r} \beta_\ell) \leq  e^{-m \theta r}  \sum_{\ell=0 }^{r} \E[ Z_{\ell}^{m} \beta_{\ell}^{-m} ]\,.
  \end{equation*}
  Then  using  (i) and (ii), we have
  \begin{align*}
 \limsup_{r \to \infty} \frac{1}{r} \log \P \bigg( \sum_{\ell=0 }^{r} Z_{\ell} > e^{\theta \ell} \bigg)
 \leq  - m \theta +  \limsup_{r \to \infty} \frac{1}{r} \log \left(  \sum_{\ell=0 }^{r} \E[ Z_{\ell}^{m}  \beta_{\ell}^{-m} ]  \right)
      \leq  - m [\theta - \log ( \lambda) ] \,,
  \end{align*}
which implies \eqref{eq-consequence} by sending $m\to \infty$.
\end{proof}

In what follows, for a rooted tree $T$  we write $Z_{m}(T)$   the number of the vertices in the   $m$-th level of $T$ and write $Z_{\leq m}(T)$ the number of vertices in the first $m$-levels.
Let $\mathbf{T}$, $\mathbf{T}'$ be  independent PGW$(\lambda)$ trees.   
 We next control the volume growth conditioned on isomorphism and also heights of the trees.

\begin{lemma}\label{cdtmoment}
Fix $\lambda>0$. For  $m\geq 0$, there exists a constant $C_{m}=C_{m}(\lambda) > 0$ such that
\begin{equation}\label{eq-volume-condition-isomorphic}
 \E( [Z_{\ell}(\mathbf{T}) ]^{m}  |  \mathbf{T} \sim_{\ell} \mathbf{T}') \leq  C_{m} \ell^{m } \mbox{ for all }\ell \geq 1\,.
\end{equation}
As a consequence,
$\E( [ Z_{ \leq \ell}(\mathbf{T}) ]^{m}   |  \mathbf{T} \sim_{\ell} \mathbf{T}') \leq  C_{m}  \ell^{2 m +2}$.
 \end{lemma}

 \begin{proof}
The main task is to prove \eqref{eq-volume-condition-isomorphic} by induction.
 The case of $m=0$ is trivial. Assume \eqref{eq-volume-condition-isomorphic} holds for all $1, \ldots,  m-1$.
Using the same argument (and notations) in the proof of Lemma \ref{Decay pr},  we have
 \begin{align*}
& a_{\ell}(m) :=\E[  Z_{\ell}(\mathbf{T})^{m} ;  \mathbf{T} \sim_{\ell} \mathbf{T}' ] = \sum_{k \geq 1} \E [ Z_{\ell}(\mathbf{T})^{m}  ; \mathbf{T} \sim_{\ell} \mathbf{T}'  | D=D'=k  ]  \mu_{k}^{2} \\
& \leq  \sum_{k \geq 1} k^2  \mu_{k}^{2} \E  \{  [Z_{\ell-1}(\mathbf{T}_{1}) + Z_{\ell}(\mathbf{T}\backslash \mathbf{T}_{1})   ]^{m} ;  \mathbf{T}_{1} \sim_{\ell-1} \mathbf{T}_{1} ' , (\mathbf{T} \backslash \mathbf{T}_{1})|_{\ell}  \sim (\mathbf{T}'\backslash \mathbf{T}'_{1})|_{\ell} |D=D'=k  \}\,.
\end{align*}
By independence among different subtrees of PGW tree, we get that  for each $0 \leq j \leq m$,
\begin{align*}
& \E  \bigg[  Z_{\ell-1}(\mathbf{T}_{1})^{j}   Z_{\ell}(\mathbf{T}\backslash \mathbf{T}_{1})^{m-j} ; \mathbf{T}_{1} \sim_{\ell-1} \mathbf{T}_{1} ' , (\mathbf{T} \backslash \mathbf{T}_{1})|_{\ell}  \sim (\mathbf{T}'\backslash \mathbf{T}'_{1})|_{\ell} \bigg| D=D'=k  \bigg]   \\
&=   \E[Z_{\ell-1}(\mathbf{T})^{j} ; \mathbf{T}  \sim_{\ell-1} \mathbf{T}'   ]  \E [  Z_{\ell}(\mathbf{T})^{m-j} ; \mathbf{T}|_{\ell}   \sim \mathbf{T}'|_{\ell}  |D=D'=k-1 ]\,.
\end{align*}
Therefore, recalling \eqref{eq-relation-mu-k} and using straightforward computations we get that
\begin{align}\label{eq-a-l-m}
a_{\ell}(m)
&\leq   \lambda^2\sum_{0 \leq j \leq m} \binom{m}{j}   \E[Z_{\ell-1}(\mathbf{T})^{j} ; \mathbf{T}  \sim_{\ell-1} \mathbf{T}'   ]  \E [  Z_{\ell}(\mathbf{T})^{m-j} ; \mathbf{T}|_{\ell}   \sim \mathbf{T}'|_{\ell} ]  \nonumber \\
&\leq \lambda^2 g_{\ell} a_{\ell-1}(m)  + \lambda^2 \sum_{1 \leq j \leq m-1} \binom{m}{j}   a_{\ell-1}(j)  \E [  Z_{\ell}(\mathbf{T})^{m-j} ; \mathbf{T}|_{\ell}   \sim \mathbf{T}'|_{\ell} ] \nonumber \\
& \phantom{ \leq \lambda^2 g_{\ell} a_{\ell-1}(m)   }  +    \lambda^2  \mathfrak{p}_{\ell-1} \left( \E [  Z_{\ell}(\mathbf{T})^{m} ; \mathbf{T} \sim \mathbf{T}'  ] + a_{\ell}(m)   \right) \,,
\end{align}
where $g_\ell$ is defined as in \eqref{eq-def-g-r}.

Note that our goal is to provide an upper bound of $ a_{\ell}(m) /\mathfrak{p}_{\ell} $. To this end, note that
 \begin{enumerate}[(1)]
 \item  $ a_{\ell-1}(j)   \lesssim_{\lambda}  (\ell-1)^{j}\mathfrak{p}_{\ell-1}$ by the induction hypothesis for $1 \leq j \leq m-1$;
 \item $ \E [  Z_{\ell}(\mathbf{T})^{m-j} ; \mathbf{T}|_{\ell}   \sim \mathbf{T}'|_{\ell} ] \lesssim_{\lambda}  \ell^{m-j-1} $ for $1 \leq j \leq m-1$.
In order to see this, note that $\{\mathbf T|_\ell \sim \mathbf T'|_\ell\} \subset \{\mathbf T \sim_\ell \mathbf T'\} \cup \{\mathbf T \sim \mathbf T'\}$. Thus, we can combine the bound that  $\E [  Z_{\ell}(\mathbf{T})^{m-j} ;  \mathbf{T} \sim_{\ell} \mathbf{T}' ] \to 0$ as $\ell \to \infty$ by the induction hypothesis and Lemma~\ref{Decay pr}, as well as the bound that
  $\E [  Z_{\ell}(\mathbf{T})^{m-j} ;   \mathbf{T} \sim \mathbf{T}' ] \leq \E [  Z_{\ell}(\mathbf{T})^{m-j} ; \mathbf{T} \text{ finite} ] \lesssim  \ell^{m-j-1} $ by Lemmas~\ref{lem-monotone} and \ref{moments} (i). 
 \item $\E [  Z_{\ell}(\mathbf{T})^{m } ;   \mathbf{T} \sim \mathbf{T}' ] \leq \E [  Z_{\ell}(\mathbf{T})^{m }   ; \mathbf{T} \text{ finite} ] \lesssim \ell^{m-1}  $ by Lemmas~\ref{lem-monotone} and \ref{moments} (i).
 \end{enumerate}
Combining (1), (2), (3) with the inequality \eqref{eq-a-l-m},  we see that there exists a constant $C_0> 0$ depending on $\lambda$ and $m$ such that
  \begin{equation}\label{eq-a-ell-m-recursion}
      \frac{a_{\ell}(m)}{\mathfrak{p}_{\ell}}
      \leq
      \frac{1}{1- \lambda^2 \mathfrak{p}_{\ell-1}} \frac{\lambda^2 g_{\ell} \mathfrak{p}_{\ell-1}}{\mathfrak{p}_{\ell}} \frac{a_{\ell-1}(m) }{\mathfrak{p}_{\ell-1}} +  C_{0} \ell  ^{m-1}\,.
  \end{equation}
By Lemma~\ref{Decay pr} (note that $g_\ell \to \gamma_{\lambda}$ quickly as in \eqref{eq-g-r-bound}), we have that
\begin{equation*}
  \prod_{\ell}\frac{\lambda^2 g_{\ell} \mathfrak{p}_{\ell-1}}{\mathfrak{p}_{\ell}} \leq \limsup_{\ell \to \infty} \mathfrak p_{\ell}^{-1}\prod_{i=1}^{\ell} (\lambda^2 g_i)  < \infty \mbox{ and }\prod_{\ell}(1- \lambda^2 \mathfrak{p}_{\ell})^{-1} < \infty \,.
\end{equation*}
Combined with \eqref{eq-a-ell-m-recursion}, it yields that  there is a constant  $C_m' > 0$ depending on $m$ and $\lambda$, such that $ a_{\ell}(m)/\mathfrak{p}_{\ell} \leq C_m' \ell^{m} $.   This completes the proof for \eqref{eq-volume-condition-isomorphic}.

We next show how to derive the consequence from \eqref{eq-volume-condition-isomorphic}. 
If $ \mathbf{T} \sim_{\ell} \mathbf{T} $, then  
$ \mathbf{T} \sim_{j} \mathbf{T}' $  and there exist $ v, v'$ in  the $j$-th levels of $\mathbf{T}$, $\mathbf{T}'$ respectively such that 
$\mathbf{T}_{v} \sim_{\ell-j} \mathbf{T}_{v'}'$. Therefore we have, for $j \leq \ell$   
\begin{equation*}
  \E[  Z_{j}(\mathbf{T})^{m} 1_{ \{ \mathbf{T} \sim_{\ell} \mathbf{T}' \}}\mid  \mathbf{T}|_{j} ,\mathbf{T}'|_{j}    ]  \leq    Z_{j}(\mathbf{T})^{m}  1_{ \{ \mathbf{T} \sim_{j} \mathbf{T}' \}} \times Z_{j}(T)^{2} \times    \P(\mathbf{T}\sim_{\ell-j} \mathbf{T}) \,,
\end{equation*}
which implies that  $ \E( [Z_{j}(\mathbf{T}) ]^{m}  ; \mathbf{T} \sim_{\ell} \mathbf{T}') \leq  C_{m}' \mathfrak{p}_{j} j^{m+2} \cdot \mathfrak{p}_{\ell-j} \leq C_{m} \ell^{m+2} \mathfrak{p}_{\ell}  $ for $j \leq \ell$. Here $C_{m}= C_{m}' \sup_{\ell \geq 1, j \leq \ell } \frac{\mathfrak{p}_{j}\mathfrak{p}_{\ell-j}}{\mathfrak{p}_{\ell}} <\infty$ by  Lemma \ref{Decay pr}.
Therefore,
 \begin{align*}
  & \E( [ Z_{ \leq \ell}(\mathbf{T}) ]^{m}  ; \mathbf{T} \sim_{\ell} \mathbf{T}') =
   \sum_{1 \leq j_{1}, \ldots, j_{m} \leq \ell} \E[ Z_{ j_{1}}(\mathbf{T}) \cdots Z_{j_{m}}(\mathbf{T}) ; \mathbf{T} \sim_{\ell} \mathbf{T}' ]\\
  & \leq \sum_{1 \leq j_{1}, \ldots, j_{m} \leq \ell} \prod_{k=1}^{m} \E[  Z_{j_{k}}^{m}(\mathbf{T}) ; \mathbf{T} \sim_{\ell} \mathbf{T}' ]^{1/m} \leq \sum_{1 \leq j_{1}, \ldots, j_{m} \leq \ell} C_{m}  \ell^{m +2 } \mathfrak{p}_{\ell}=  C_{m} \ell^{2m +2 }\mathfrak{p}_{\ell}  \,, 
 \end{align*}
  yielding that  $\E( [ Z_{ \leq \ell}(\mathbf{T}) ]^{m}  |\mathbf{T} \sim_{\ell} \mathbf{T}')  \leq  C_{m} \ell^{2m +2 }$.
 \end{proof}

We next control the volume of PGW trees conditioned on isomorphism without constraints on  heights of the trees. Our bound is likely far from being sharp, as suggested by \cite[Theorem 2]{Olsson22} for a bound of exponential decay when the offspring distribution has finite support.

\begin{lemma}\label{degree}
  Let $\{\xi_i:  1 \leq i \leq m\}$  be i.i.d.\ Poisson $(\lambda)$ random variables. Then for sufficiently large $m$,
  \begin{equation*}
    \P \left(  \{\xi_{k}: 1\leq k\leq m\} = \{\xi'_{k}: 1\leq k\leq m\}  \right) \leq    \exp\{ - (\log m)^{ 1.7 } \}  \,.
  \end{equation*}
  \end{lemma}

  \begin{proof}
Let $M$ and $M'$ be two independent Poisson variables with mean $m$.
In addition, let $N_{k}= \sum_{i=1}^{M} 1_{\{\xi_{i} =k\}} $  and $N'_{k}= \sum_{i=1}^{M'} 1_{\{\xi'_{i} =k\}} $ for all $k \geq 0$.
 By Poisson thinning property, we see that $N_k$ and $N'_k$ for $k = 0, 1, \ldots$ are mutually independent Poisson variables with
$\E  N_k = \E N'_k = m \mu_k$ (recall that $\mu_k = \P(\xi_1 = k)$). A simple computation gives that $\P(M = M' = m) \geq c/m$ for a positive constant $c>0$. Therefore,
\begin{align*}
\P &\left(  \{\xi_{k}: 1\leq k\leq m\} = \{\xi'_{k}: 1\leq k\leq m\}  \right) = \P \left( N_k = N'_k \mbox{ for } k \geq 0 \mid M = M' = m \right)\\
&\leq  O(m) \P \left( N_k = N'_k \mbox{ for } k \geq 0 \right) = O(m) \prod_{k=0}^\infty \P(N_k = N'_k) \leq O(m) e^{-(\log m)^{1.8}}
\end{align*}
where the last inequality follows from a straightforward bound on $\P(N_k = N'_k) \leq m^{-0.01} $ for $k\leq \frac{\log m}{100 \log \log m}$ (the power $1.8$ is chosen as a arbitrary number less than 2). This completes the proof of the lemma.
  \end{proof}

\begin{lemma}\label{isodecay}
For $\lambda \geq 1$,  let $\mathbf{T}, \mathbf{T}'$ be two independent PGW$(\lambda)$ trees. Then  we have for sufficiently large $m$,
  \begin{equation*}
  \P(  \mathbf{T} \sim \mathbf{T}' ;   |\mathbf{T}| \geq  m) \leq \exp\{ -  (\log m)^{3/2} \}\,.
  \end{equation*} 
  \end{lemma}

  \begin{proof}
  Indeed,  for $\lambda > 1$ we have exponential decay in $m$, since $\{\mathbf{T} \sim \mathbf{T}'\}$ implies that $\mathbf{T}$ is a finite tree (Lemma~\ref{lem-monotone}). However,  we prove the above results for  all $\lambda \geq 1$.

  Note that $\mathbf{T} \sim \mathbf{T}' $ implies that the degree sequence  $(\xi_{i})_{i=1}^{|\mathbf{T}|}$ and $(\xi'_{i})_{i=1}^{|\mathbf{T}|}$ has the same empirical distribution,
  where $\xi_{i}$  is the number of the descendants in the tree  $\mathbf{T}$ of the $i$-th vertex in the tree $\mathbf{T}$ (here we may use  the breadth-first  order, but the ordering is irrelevant anyway since we are only interested in the empirical distribution). Applying Lemma \ref{degree}, we have,
  \begin{align*}
       \P( \mathbf{T} \sim \mathbf{T}' ; |\mathbf{T}| \geq  m)&= \sum_{t \geq m} \P( \mathbf{T} \sim \mathbf{T}' ; |\mathbf{T}| = t)  \leq \sum_{t \geq m} \P(   \{\xi_{k}: 1\leq k\leq t\} = \{\xi'_{k}: 1\leq k\leq t\}) \\
       & \leq \sum_{t \geq m}  \exp\{ -  (\log t)^{1.7} \}  \leq   \exp\{ -  (\log m)^{3/2} \}
  \end{align*}
  for large $m$. This completes the proof of the lemma.
  \end{proof}

In the end of this subsection, we prove some tail estimates for binomial variables. There is nothing novel, and we only record the proof for completeness.

\begin{lemma}\label{lemma-binomial-tail}
  Let  $X= \sum_{i=1}^{m} X_{i}$ where $X_i$'s are independent  and $X_i$ is a Bernoulli variable with parameter $p_i$ for $1\leq i\leq m$. Then for all  $x>0$,  we have
    \begin{equation*}
    \mathbb{P}(X \geq x) \leq  \left(\frac{e\mathbb{E} [X]}{x}\right)^{x} \,.
    \end{equation*}
    \end{lemma}

    \begin{proof}
     For any $\theta>0$, a direct computation yields
     \begin{align*}
     \mathbb{P}(X>x) &\leq e^{-\theta x} \mathbb{E} [e^{\theta X} ]  =  \exp \left\{ \sum_{i=1}^{m} \log (1+ p_{i} (e^{\theta}-1) )-\theta x\right\}   \\
     & \leq   \exp \left\{\mathbb{E} X\left(e^{\theta}-1\right)-\theta x\right\}\,.
     \end{align*}
     Setting $\theta=\log (1+x / \mathbb{E} [X])$ in the previous inequality, we get that
     \begin{equation*}
     \mathbb{P}(X>x) \leq\left(\frac{e \mathbb{E} [X]}{x+\mathbb{E} [X]}\right)^{x}  \leq\left(\frac{e}{x}\right)^{x}(\mathbb{E} X)^{x}\,. \qedhere
     \end{equation*}
  \end{proof}

  \begin{lemma}\label{lemma-sum-of-binomial}
   Assume $1\leq f(n) = o(\sqrt{n})$. Let $X_{1} \sim \mathrm{Bin}(f(n)^2, \frac{\lambda}{n})$.
 Let $  X_2 =\sum_{j=1}^{rn} \xi_{1,j} \xi_{2,j}$, where $(\xi_{i,j})$ are i.i.d. $\mathrm{Bin}(f(n), \frac{\lambda}{n})$ variables and are independent of $X_{1}$.
Then  $ \P(  X_1+X_2 \geq k  ) \lesssim_{\lambda}  \left( \frac{r f(n)^{2}}{n} \right)^{k}$ for $k=1,2,3$.
 \end{lemma}

 \begin{proof}
Let $X_{3}=  \sum_{j=1}^{rn} 1_{\left\{ \xi_{1,j} \geq 1 \right\}}  1_{\left\{ \xi_{2,j} \geq 1 \right\}}$. Then $X_3 \sim \mathrm{Bin}(rn, p_{n}^{2})$ where $p_{n}=\P(\xi \geq 1)  = (1+o(1)) \frac{\lambda f(n)}{n}$ (here we write  $\xi=\xi_{1,1}$ for short).
Thus,
 \begin{align}\label{eq-decomposition-probability-lemma-4.8}
   \P(  X_1+X_2 \geq 3  ) &\leq
      \P(  X_1+X_3 \geq 3  )  + \P(   X_2 \geq 3 ,  X_3 = 2  )  \nonumber\\
& \quad +  \P(   X_2 \geq 3 ,  X_3 = 1  ) + \P(   X_2 \geq 2 ,  X_3 = 1  ) \P(X_1 \geq 1) \,.
 \end{align}
By Lemma \ref{lemma-binomial-tail}, we have $\P(X_1 + X_3 \geq 3) \lesssim_\lambda (f(n)^2/n)^3 $ and in addition  $
  \P( \xi \geq k | \xi \geq 1  ) \lesssim_\lambda  (f(n)/n)^{k-1}$ for $k\geq 1$. Then,
 \begin{align*}
     \P(   X_2 \geq 3 |  X_3 = 1  )&=\P( \xi_{1,1} \xi_{2,1} \geq 3 |\xi_{1,1} \geq 1, \xi_{2,1} \geq 1  )\\
    & \leq 2 \P(  \xi   \geq 3 |  \xi \geq 1   ) + \P(  \xi   \geq 2 |  \xi \geq 1   )\P(  \xi   \geq 2 |  \xi \geq 1   ) \lesssim_{\lambda} \frac{f(n)^{2}}{n^{2}} \,.
     \end{align*}
   Similarly we have $\P(   X_2 \geq 2 |  X_3 = 1  ) \lesssim  \frac{f(n)}{n}$. Moreover, $\P(X_3 = 2) \lesssim_{\lambda} ((f(n))^2/n)^2$ and
   \begin{align*}
      \P(   X_2 \geq 3 |  X_3 = 2  )&=\P( \xi_{1,1} \xi_{2,1} + \xi_{1,2} \xi_{2,2}   \geq 3 |\xi_{i,j} \geq 1; i,j \leq 2  )\\
     & \leq 4 \P(  \xi \geq 2 |  \xi \geq 1   )   \lesssim_{\lambda} \frac{f(n)}{n} \,.
      \end{align*}
Plugging all these estimates into \eqref{eq-decomposition-probability-lemma-4.8} (together with straightforward bound on $\P(X_1 \geq 1)$ and $\P(X_3=1)$ as well as $\P(X_3 = 2)$),  we have $ \P(  X_1+X_2 \geq 3  ) \lesssim_{\lambda}  \frac{r^3 f(n)^{6}}{n^3} $. The same argument shows  $ \P(  X_1+X_2 \geq 2   ) \lesssim_{\lambda}  \frac{r^2 f(n)^{4}}{n^2} $ (and a much simpler arguments proves the bound for $\P(  X_1+X_2 \geq 1)$).
     \end{proof}

 \subsection{Proof of Lemma~\ref{lemma-reducedBFS}} \label{sec:lem-BSF}

This subsection is devoted to the proof of Lemma \ref{lemma-reducedBFS}. To this end, we need the following four lemmas whose proofs are presented at the end of this subsection.

 \begin{lemma}\label{lemma-complexity-of-Nr}
  For the Erd\H{o}s-R\'{e}nyi graph $\mathcal{G}=\mathcal{G}_{n,\frac{\lambda}{n}}$,  there exist constants $\delta =\delta_{\lambda, \epsilon_0}$ and $s_{\lambda}=s_{\lambda, \epsilon_0}$ depending only on $\lambda$ and $\epsilon_0$ such that for all $u, v\in \mathcal G$,
  \begin{align}
\P( |\mathsf{N}_{r}(v)|  > n^{\frac{1-\delta}{2}} ) &  =    o(n^{-2})\,;\label{eq-i-volume}\\
 \P( \mathrm{Comp}(\mathsf{N}_{r}(v))  >  s_{\lambda})& = o(n^{-2})\,;\label{eq-prob-complexity-bound}\\
 \P(\Xi_{2}(u,v) > s_{\lambda}) &= o(n^{-2})\,. \label{eq-Xi-2-bound}
  \end{align}
   \end{lemma}

The next lemma  needs  some notations. If we only keep the relative ordering  for labels  of the vertices in $\mathcal{T}_{\mathrm{aux}}(u)$, we get a  rooted ordered tree. More precisely,  let  $\mathcal{U}=\cup_{n=0}^{\infty} \mathbb{N}^{n}$. We map $y \in \mathcal{T}_{\mathrm{aux}}(u)$ to a label $\mathbf{i}=i_{1} i_{2} \ldots i_{m} \in \mathcal{U}$, if the path from $u$ to $y$ is  $\left(y_0=u, y_{1}, \ldots, y_{m}=y\right)$ and $y_{k}$ is the $i_{k}$-th smallest one among all the children of $y_{k-1}$.   Then we get a rooted order tree which we denote by $\tau$ (we  write $\mathcal{T}_{\mathrm{aux}}(u)=\tau$ for short).

 \begin{lemma}\label{lemma-aux-tree-vs-GW-tree}
  Let  $\delta $ be chosen as in Lemma \ref{lemma-complexity-of-Nr}.
 For  $u, v \in \mathcal{G}_{n,\frac{\lambda}{n}}$,  we have
  \begin{equation*}
    \P(  | \mathcal{T}_{\mathrm{aux}}(u)|_{r} | >      n^{\frac{1-\delta}{2}} \text{ or }  | \mathcal{T}_{\mathrm{aux}}(v) |_{r} |  >    n^{\frac{1-\delta}{2}} ) = o(n^{-2}) \,.
      \end{equation*}
  Moreover, there exists $\Delta_{n} $ depending only on $\delta$ with  $\Delta_{n}  \overset{n \to \infty}{\longrightarrow}  0$ such that for all  rooted ordered trees $\tau , \tau'$ with $|\tau|, |\tau| \leq   n^{\frac{1-\delta}{2}}$, we have   (denote by $\mathbf{T}, \mathbf{T}'$ two independent PGW($\lambda$) trees)
    \begin{equation*}
      \Big| \frac{\P( \mathcal{T}_{\mathrm{aux}}(u)|_{r}  = \tau, \mathcal{T}_{\mathrm{aux}}(v)|_{r}  = \tau')}{\P( \mathbf{T}|_{r}   =    \tau )\P( \mathbf{T}|_{r}    = \tau'   )}  -1  \Big| \leq \Delta_n \,.
    \end{equation*}
  \end{lemma}

 \begin{lemma}\label{lemma-condi-isomorphic-prob-aux-tree} Let  $s_{\lambda}$ be chosen as in Lemma \ref{lemma-complexity-of-Nr}.  For two vertices $u$ and $v$, let $\Omega_{a} =\{ \mathsf{N}_{r+1}(u) \sim \mathsf{N}_{r+1}(v), \mathrm{Comp}(\mathsf{N}_{r}(u)) \leq s_{\lambda} , \Xi_{2} \leq s_{\lambda}\}$. Then, there exists a constant $C_\lambda > 0$ depending on $\lambda$ and $s_{\lambda}$ such that for any event $\Omega_{b}$ which is measurable  with respect to the $\sigma$-field generated by $ \{ \Xi_{i},\Lambda_{i} \mbox{ for } i =1,2 $,  $|\mathcal{T}_{\mathrm{cut}}(u)|, |\mathcal{T}_{\mathrm{cut}}(v)|$, $H(\mathcal{T}_{\mathrm{cut}}(u)),  H(\mathcal{T}_{\mathrm{cut}}(v))$, $|\mathsf{N}_{r}(w)| \mbox{ for } w \in \mathcal{G} \}$, we have that
 \begin{equation*}
  \P( \mathcal{T}_{\mathrm{aux}}(u)|_{r} \sim  \mathcal{T}_{\mathrm{aux}}(v)|_{r}  \big|  \Omega_{a} \cap \Omega_{b} ) \geq \frac{1}{C_{\lambda}} \,.
 \end{equation*}
 \end{lemma}

 \begin{lemma}\label{lemma-size-of-bridging-tree}
Let  $\delta $ be chosen as in Lemma \ref{lemma-complexity-of-Nr}. For any two vertices  $u,v \in \mathcal G_{n, \frac{\lambda}{n}}$, we have that 
\begin{equation*}
  \P(\mathsf{N}_{r+1}(u) \sim \mathsf{N}_{r+1}(v)  \mbox{ and } |\mathcal{T}_{\mathrm{cut}}(u)|_{r}|=|\mathcal{T}_{\mathrm{cut}}(v)|_{r}| \geq n^{\frac{\epsilon_0 \wedge \delta }{9}}) = o(n^{-2})\,.
\end{equation*} 
 \end{lemma}

\begin{proof}[Proof of Lemma \ref{lemma-reducedBFS}]
  Let $\Omega_{\mathrm{typ}}$ be the intersection of the following events: $\mathsf{N}_{r+1}(u) \sim \mathsf{N}_{r+1}(v) $,  $\mathsf{N}_{r}(u)$ survives,  $|\mathcal{T}_{\mathrm{cut}}(u)|_{r}| \leq f(n)$ where $f(n)=n^{\frac{\epsilon_0 \wedge \delta }{9}}$,
  $\mathrm{Comp}(\mathsf{N}_{r}(u)) \leq s_{\lambda} $ and $|\mathsf{N}_{r}(w)|\leq n^{\frac{1-\delta}{2}}$ for all $w \in \mathcal{G}$. Then by Lemmas~\ref{lemma-complexity-of-Nr} and \ref{lemma-size-of-bridging-tree}, we have 
  \begin{equation*}
  \P( \{ \mathsf{N}_{r+1}(u) \sim \mathsf{N}_{r +1}(v), \mathsf{N}_r(u)  \text{ survives}\}  \cap \Omega_{\mathrm{typ}}^{c}   )= o(\frac{1}{n^2})\,.
  \end{equation*} 
  For an event $A$, we define 
 $\P_{\mathrm{typ}}( A) = \P ( A \cap \Omega_{\mathrm{typ}} )$. Next, we prove Lemma~\ref{lemma-reducedBFS} item by item.

  \noindent (i). On  the event  $  |A_{\leq r}(v)|= |A_{\leq r}(u)| \leq f(n)$, we have that
  $\Xi_{1}=\sum_{x \in A_{\leq r}(u), y \in A_{\leq r}(v)} \mathcal{G}_{xy} $ is stochastically dominated by a binomial variable $ X_ 1\sim \mathrm{Bin}(f(n)^{2}, \frac{\lambda}{n}) $.
   Similarly,  since $ \Xi_{2} =\sum_{t=0}^{r-1} \sum_{w \in U_{t}} \mathcal{G}_{w,A_{t}(u)}\mathcal{G}_{w,A_{t}(v)}$,  we have  $\Xi_{2}$ is stochastically dominated by $X_2 =\sum_{j=1}^{rn} \xi_{1,j} \xi_{2,j}$, where $\xi_{i,j}$'s are i.i.d.\ binomial variables $\mathrm{Bin}(f(n), \frac{\lambda}{n})$.
Observing that $\Xi_{1}$ and $\Xi_{2}$ are independent since they are measurable  functions of different edges,  $\Xi $ is stochastically dominated by $X_1 + X_2$ with $X_1$ independent of $X_2$. Applying Lemma \ref{lemma-sum-of-binomial}, we have
    \begin{align*}
    & \P_{\mathrm{typ}}(   \Xi \geq 3    )   \leq   \P(  X_1+X_2 \geq 3  ) \lesssim_{\lambda}  \frac{ r^3 f(n)^{6}}{n^3} = o (\frac{1}{n^{2}}) \,.
    \end{align*}
Combined with \eqref{eq-ii-first-bound}, this proves the first assertion in (i) via a simple union bound.

Furthermore, when $\mathsf{N}_{r+1}(u) \sim \mathsf{N}_{r+1}(v) $,  if $w \in U_{t}$ and  $\mathcal{G}_{w,A_{t}(u)} \mathcal{G}_{w,A_{t}(v)}\geq 1$, then by \eqref{eq-Av-onto-Au} we have either  $\mathcal{G}_{w,A_{t}(u)} =\mathcal{G}_{w,A_{t}(v)}$ or there exists $w' \neq w$ in $U_{t}$ such that $\mathcal{G}_{w,A_{t}(u)} \mathcal{G}_{w,A_{t}(v)} = \mathcal{G}_{w',A_{t}(u)} \mathcal{G}_{w',A_{t}(v)}$. Thus when  $\Xi_{2}    \leq 2$,  we have $\mathcal{G}_{w,A_{t}(u) } \mathcal{G}_{w,A_{t}(v) } \in \{0,1 \}$ for all $w \in U_{t}$. Hence $|R_{\leq r}| = \Xi_{2}$.

 \noindent  (ii).  On the event $\Omega_{\mathrm{typ}}$  and $\Xi = 0$, by definition we have $\Lambda_{i}=0$ for $i=2,3$. The trees $\mathcal{T}_{\mathrm{cut}}(v)$, $\mathcal{T}_{\mathrm{cut}}(u)$  are exactly the auxiliary trees $\mathcal{T}_{\mathrm{aux}}(v)$, $\mathcal{T}_{\mathrm{aux}}(u)$, and  $H(\mathcal{T}_{\mathrm{aux}}(v)) \geq r$. Applying  Lemma \ref{lemma-condi-isomorphic-prob-aux-tree}, we have
   \begin{align}\label{eq-isomorphic-lower-conditional-prob}
   \P( \mathcal{T}_{\mathrm{aux}}(v)  \sim_{r} \mathcal{T}_{\mathrm{aux}}(u)  | \Omega_{\mathrm{typ}} ,  \Xi=0 , \Lambda_{1} \geq 1 ) & \geq \frac{1}{C_{\lambda}}  \,.
    \end{align}
   On the other hand,  note that $\mathrm{Comp}(\mathcal{G}[A_{\leq r}(u)])$  comes from  edges  within $A_{\leq r}(u)$ but not in $E(\mathcal{T}_{\mathrm{cut}}(u)|_{r})$. 
Conditioned on $ \mathcal{T}_{\mathrm{cut}}(u)|_{r}$ and $\mathcal{T}_{\mathrm{cut}}(v)|_{r} $ with $|\mathcal{T}_{\mathrm{cut}}(u)|_{r}| \leq f(n) $ and $ |\mathcal{T}_{\mathrm{cut}}(v)|_{r}| \leq f(n)$, we have that $\Lambda_{1}(u) $ and $\Lambda_{1}(v)$ are independent and are both  stochastically dominated by a binomial variable $Y_{1} \sim \mathrm{Bin}( f(n)^{2} , \frac{\lambda}{n})$.
 By the fact that  $\Lambda_{1}(u)=\Lambda_{1}(v)$  on the event $\mathsf{N}_{r+1}(u) \sim \mathsf{N}_{r+1}(v) $ and
   by \eqref{eq-isomorphic-lower-conditional-prob},
   we have
   \begin{align}
   & \P_{\mathrm{typ}}( \Xi=0,  \Lambda_{1} \geq 1 )  \leq  C_{\lambda}\P_{\mathrm{typ}}( \Xi=0, \mathcal{T}_{\mathrm{aux}}(u) \sim_{r} \mathcal{T}_{\mathrm{aux}}(v),   \Lambda_{1}(u) \geq 1,   \Lambda_{1}(v) \geq 1) \nonumber \\
   & \leq C_{\lambda} \P( \mathcal{T}_{\mathrm{aux}}(u) \sim_{r} \mathcal{T}_{\mathrm{aux}}(v)  ) \P( Y_{1}\geq 1 )^{2}   \lesssim_{\lambda}   \alpha_{\lambda}^{r} \frac{f(n)^2}{n} = o(\frac{1}{n^2}) \,, \label{eq-ii-first-bound}
   \end{align}
where the second inequality follows from independence and the aforementioned stochastic dominance,
 and the third inequality follows from Lemmas \ref{lemma-aux-tree-vs-GW-tree}, \ref{Decay pr} and \ref{lemma-binomial-tail}.

  When $\Xi= \Lambda =0$,  we have $\mathsf{N}_{r}(u)= \mathcal{T}_{\mathrm{aux}}(u)|_{r}$. If $u$ has two $r$-arms, then the tree $\mathcal{T}_{\mathrm{aux}}(u)$ has two subtrees both of which survive $(r-1)$ levels. Hence  by Lemmas \ref{lemma-aux-tree-vs-GW-tree} and \ref{Decay pr},
    \begin{align*}
   & \P_{\mathrm{typ}}(   \Lambda =\Xi=0, u \text{ has two $r$-arms} ) \\
    & \leq   \P( \mathcal{T}_{\mathrm{aux}} (u)  \sim_{r} \mathcal{T}_{\mathrm{aux}} (v), \mathcal{T}_{\mathrm{aux}}(u) \text{ has two  subtrees surviving $(r-1)$ levels} ) \\
     & \lesssim_{\lambda}   \alpha_{\lambda}^{r} \times \alpha_{\lambda}^{r} =  o(\frac{1}{n^2}) \,.
    \end{align*}
    Combined with \eqref{eq-ii-first-bound}, this proves (ii) via a simple union bound.

 \noindent  (iii). When  $\Xi=\Xi_{1}=1$, by definition  $\Lambda_{2}=0$. 
 We next consider the case for $\Xi=\Xi_{2}=1$.  
Since  $\Lambda_{2}(u)$, $\Lambda_{2}(v)$ and $\Xi$ are measurable  functions of different edges,  conditioned on $\{ \Xi=\Xi_{2}=1 , |A_{\leq r}(u)|= |A_{\leq r}(v)| \leq f(n), |\mathsf{N}_{r}(w)| \leq n^{\frac{1-\delta}{2}} \text{ for } w \in R_{\leq r} \}$, we have that $\Lambda_{2}(u)$ and $\Lambda_{2}(v)$ are independent and  are both  stochastically dominated by a binomial variable $Y_{2} \sim \mathrm{Bin}(f(n)n^{\frac{1- \delta}{2}}, \frac{\lambda}{n})$.  Noting that $\Lambda_{2}(u)=\Lambda_{2}(v)$ when $\mathsf{N}_{r+1}(u) \sim \mathsf{N}_{r+1}(v) $, by  Lemmas \ref{lemma-sum-of-binomial}, \ref{lemma-binomial-tail}, we have for $i=1,2$,
   \begin{align*}
    \P_{\mathrm{typ}}( \Xi=1, \Lambda_{i} \geq 1 ) &=   \P_{\mathrm{typ}}( \Xi=1, \Lambda_{i}(u) \geq 1,\Lambda_{i}(v) \geq 1 ) \\
   &\leq    \P( X_{1}+ X_{2} \geq 1   ) \P( Y_{i} \geq 1 )^{2} = o(\frac{1}{n^2})\,.
    \end{align*}
Write $\rho'= \frac{1+\epsilon_0/2}{1+\epsilon_0}$. Note that $\mathcal{T}_{\mathrm{cut}}(u)$ is a subtree of $\mathcal{T}_{\mathrm{aux}}(u)$ by deleting at most $1$ vertex when $\Xi = 1$.  Therefore,
\begin{align*}
 & \P_{\mathrm{typ}}( \Xi=1,  \Lambda=0, H( \mathcal{T}_{\mathrm{cut}}(v)) >  \rho' r ) \leq  C_{\lambda}\P_{\mathrm{typ}}( \Xi=1,   \mathcal{T}_{\mathrm{aux}}(v) \sim_{ \rho' r }  \mathcal{T}_{\mathrm{aux}}(u))   \\
& \leq  C_{\lambda} \P(  \mathcal{T}_{\mathrm{aux}}(v) \sim_{ \rho' r }  \mathcal{T}_{\mathrm{aux}}(u) )   \P( X_{1}+ X_{2} \geq 1  )
 \lesssim_{\lambda}   r^{3} \alpha_{\lambda}^{ \rho' r}    \frac{rf(n)^{2}}{n} = o (\frac{1}{n^2}) \,.
 \end{align*}
Here the first inequality follows from Lemma~\ref{lemma-condi-isomorphic-prob-aux-tree}; the second inequality follows from the independence and the stochastic dominance;  the third inequality  follows from Lemmas~\ref{lemma-aux-tree-vs-GW-tree}, \ref{lemma-sum-of-binomial} and \ref{Decay pr}. This proves (iii) via a simple union bound.

\noindent (iv).  On $\Omega_{\mathrm{typ}} $, when $|R_{\leq r}|\leq 1$ we have   $\Lambda_{3}=0$. When $\Xi= \Xi_{2}=|R_{\leq r}|=2$ (and we write $R_{\leq r} = \{w_1, w_2\}$),  
we have  $\Lambda_{3} \geq 1$ only if there exists $\ell \leq r$ such that $ \mathsf N_{k}(w_{1}; \mathcal{G}[V\backslash A_{\leq r}] )  \cap \mathsf N_{k}(w_{2}; \mathcal{G}[V\backslash A_{\leq r} ]) = \emptyset$ holds for $k=\ell-1$ but not $k=\ell$. Here $\mathcal{G}[A]$ is the subgraph on $\mathcal{G}$ induced by the vertex set $A$. 
Conditioned on $\Xi=\Xi_{2}=2$  and $|\mathsf N_{k}(w_{i}; \mathcal{G}[V\backslash A_{\leq r}] )| \leq n^{\frac{1-\delta}{2}}$ for $i \in \{1, 2\}$, we have that  $\Lambda_{3} $ is stochastically dominated by some $ Y_{3} \sim \mathrm{Bin} (n^{1- \delta}, \frac{\lambda}{n})$.
 Then by  the aforementioned stochastic dominance and Lemmas~\ref{lemma-sum-of-binomial}, \ref{lemma-binomial-tail}, we have that for $i=1,2,3$,
    \begin{align}\label{eq-iv-Lambda-larger-0}
     \P_{\mathrm{typ}}( \Xi=2, \Lambda_{i} \geq 1) \leq  \P( X_{1}+X_{2} \geq 2  )\P( Y_{i} \geq 1 )= o(\frac{1}{n^2})  \,.
    \end{align}
In addition, by Lemma \ref{lemma-condi-isomorphic-prob-aux-tree},  we have that for every $\ell \leq r$,
  \begin{equation}\label{eq-aux-isomorphic-conditional-prob}
 \P( \mathcal{T}_{\mathrm{aux}}(u)  \sim_{\ell} \mathcal{T}_{\mathrm{aux}}(v) | \Omega_{\mathrm{typ}}, \Xi=2, \Lambda= 0, H( \mathcal{T}_{\mathrm{cut}}(v)) =  \ell)  \geq \frac{1}{C_{\lambda}} \,.
  \end{equation}
Let $\Gamma_{ \ell}=  \sum_{t=0}^{\ell}\sum_{x \in A_{t}(u), y \in A_{t}(v)} \mathcal{G}_{xy} +     \sum_{t=0}^{\ell-1}\sum_{w \in U_{t}} 1_{\{ \mathcal{G}_{w, A_{t}(u)} \geq 1 \}} 1_{ \{\mathcal{G}_{w, A_{t}(v) \geq 1}\}}$.  By the fact that $\Gamma_\ell=2$ (when  $\Xi=2$ and $H(\mathcal T_{\mathrm{cut}}(v)) = \ell$) and the fact that
 $H(\mathcal T_{\mathrm{aux}}(v)) \geq H(\mathcal T_{\mathrm{cut}}(v))$, we get from \eqref{eq-aux-isomorphic-conditional-prob} that
 \begin{align}\label{eq-iv-Lambda-0}
 \P_{\mathrm{typ}}( \Xi=2, \Lambda= 0, H( \mathcal{T}_{\mathrm{cut}}(v)) >  L)   \leq C_{\lambda} \sum_{\ell = L}^{r} \P(  \mathcal{T}_{\mathrm{aux}}(v) \sim_{\ell}  \mathcal{T}_{\mathrm{aux}}(u) ; \Gamma_{\ell}  = 2 )   \,,
 \end{align}
Given $\mathcal{T}_{\mathrm{cut}}(u)$ and $\mathcal{T}_{\mathrm{cut}}(v)$ as well as $\mathcal{T}_{\mathrm{aux}}(u)$ and $\mathcal{T}_{\mathrm{aux}}(v)$,  we have that $\Gamma_{\ell}$ is stochastically dominated by the sum of  $\mathrm{Bin}( |\mathcal{T}_{\mathrm{cut}}(u)|_{\ell}||\mathcal{T}_{\mathrm{cut}}(u)|_{\ell}|, \frac{\lambda}{n}) $  and $ \mathrm{Bin}( n, \frac{\lambda |\mathcal{T}_{\mathrm{cut}}(u)|_{\ell}||\mathcal{T}_{\mathrm{cut}}(u)|_{\ell}| }{n^2}) $ where these two binomial variables are independent. Then by  Lemmas \ref{lemma-aux-tree-vs-GW-tree} and \ref{lemma-binomial-tail},
\begin{align*}
&    \sum_{\ell = L}^{r}   \E \left[   \P( \Gamma_{\ell}  = 2 | \mathcal{T}_{\mathrm{aux}}(u), \mathcal{T}_{\mathrm{aux}}(v) )  ;   \mathcal{T}_{\mathrm{aux}}(u) \sim_{\ell} \mathcal{T}_{\mathrm{aux}}(v)  \right] \\
& \lesssim_{\lambda}    \sum_{\ell = L}^{r}   \E \left[  \frac{| Z_{\leq \ell}( \mathcal{T}_{\mathrm{aux}}(u))|^{4}}{n^2}  ;   \mathcal{T}_{\mathrm{aux}}(u) \sim_{\ell} \mathcal{T}_{\mathrm{aux}}(v) ; | \mathcal{T}_{\mathrm{aux}}(u)|_r| \leq n^{\frac{1-\delta}{2}} \right]  +o(n^{-2}) \\
&  \lesssim_{\lambda}   \frac{1}{n^2}   \sum_{\ell = L}^{\infty}   \E \left[ | Z_{\leq \ell}(\mathbf T)|^{4}  ;   \mathbf T  \sim_{\ell} \mathbf T'  \right] + o(n^{-2})\,,
\end{align*}
where $Z_{\leq \ell}$ (as before) denotes for the number of vertices in the first $\ell$-levels of a tree and $\mathbf T, \mathbf T'$ are two independent PGW($\lambda$)-trees.
Thanks to Lemmas~\ref{cdtmoment} and \ref{Decay pr}, we have that the sum on the right-hand side above vanishes in $L$. 
Combined with \eqref{eq-iv-Lambda-larger-0} and \eqref{eq-iv-Lambda-0}, it yields (iv)  by a simple union bound.
\end{proof}

It remains to provide the postponed proofs for Lemmas \ref{lemma-complexity-of-Nr}, \ref{lemma-aux-tree-vs-GW-tree},   \ref{lemma-aux-tree-vs-GW-tree} and \ref{lemma-condi-isomorphic-prob-aux-tree}.

 \begin{proof}[Proof of Lemma \ref{lemma-complexity-of-Nr}]
 We first prove (i). When $\lambda < 1$, with high probability every component in $\mathcal{G}$ has size $O(\log n)$ and has at most one cycle (see, e.g., \cite{Bollobas01}). The desired quantitative bounds are also straightforward in this case and can be proved via standard methods (or by an easy adaption for the arguments below when $\lambda \geq 1$). Thus, in what follows we focus on $\lambda \geq 1$.  We first control the volume of $\mathsf N_r(v)$ for $v\in \mathcal G$.
Employing a standard breadth-first-search algorithm, we see that  $|\mathsf{N}_{r}(v)|$ is stochastically dominated by the number of vertices in the first $r$-levels of  a $\mathrm{Bin}(n, \frac{\lambda}{n}) $-GW branching process (denoted as $Z_{\leq r}$).
 By (the second inequality in) \eqref{eq-rho-r-identifiable},  there exists $\delta=\delta(\lambda; \epsilon_0)$ so that   $\frac{(1-\delta)\log(\alpha_{\lambda}^{-1})}{2(1+\epsilon_0)  } > \log(\lambda) $. As $n^{\frac{1-\delta}{2}} =   \exp \{ \frac{(1-\delta)\log(\alpha_{\lambda}^{-1})}{2(1+\epsilon_0)  } r \} $, applying Lemma \ref{moments} (and Remark \ref{rmk-moments})
 we deduce \eqref{eq-i-volume} as follows:
 \begin{equation*}
 \P( |\mathsf{N}_{r}(v)|  > n^{\frac{1-\delta}{2}} ) \leq      \P \big(   Z_{\leq r} >  n^{\frac{1-\delta}{2}}   \big)  =    o(n^{-2})\,.
 \end{equation*}

A cycle in $\mathsf{N}_{r}(v)$ is created by a ``self-intersection",
that is,  an edge from $v_{t}$ (the vertex we are exploring in the breadth-first-search process at time $t$) to some vertex we have explored before $t$.
When $|\mathsf{N}_{r}(v)| \leq n^{\frac{1-\delta}{2}}$,  it is easy to see that $\mathrm{Comp}(\mathsf{N}_{r}(v) ) $ is stochastically dominated by $\mathrm{Bin}\left( n^{1-\delta}, \frac{\lambda}{n}\right)$. By Lemma \ref{lemma-binomial-tail},
  \begin{equation*}
\P\left(\mathrm{Comp}(\mathsf{N}_{r}(v) )  \geq \frac{4}{\delta}\right) \leq
 \P \left( \operatorname{Bin} ( n^{1-\delta}, \frac{\lambda}{n} )    \geq  \frac{4}{\delta}  \right) + o (\frac{1}{n^2}) =  o (\frac{1}{n^2})\,. \end{equation*}
 Taking $s_{\lambda}> 1 + \frac{4}{\delta}$ and combining \eqref{eq-i-volume},  we obtain \eqref{eq-prob-complexity-bound}.

We next prove (ii). When $| \mathsf{N}_{r}(u)|,| \mathsf{N}_{r}(v)| \leq n^{\frac{1-\delta}{2}}$, since $ \Xi_{2} =\sum_{t=0}^{r-1} \sum_{w \in U_{t}} \mathcal{G}_{w,A_{t}(u)}\mathcal{G}_{w,A_{t}(v)}$ we have $ \Xi_{2}$ is stochastically dominated by $\sum_{j=1}^{rn} \xi_{j,1}\xi_{j,2}$ where $\xi_{j,i}$'s  are i.i.d.\ binomial variables $\mathrm{Bin}(n^{\frac{1-\delta}{2}},\frac{\lambda }{n})$. Applying Lemma \ref{lemma-binomial-tail}, there exists a constant $S_{1}$ 
depending only on $\delta=\delta(\lambda;\epsilon_0)$ such that
\begin{equation*}
 \P \left(  \sum_{j=1}^{rn} 1_{\left\{  \xi_{j,1} \geq 1  \right\}}  1_{  \left\{ \xi_{j,2} \geq 1  \right\}}\geq S_{1} \right) = o (\frac{1}{n^2}) \,.
\end{equation*}
  Take $S_{2}$ large enough such that $
   \P \left(  \xi_{1,1} \geq  S_{2} \mid  \xi_{1,1} \geq 1  \right) = o (\frac{1}{n^2})$,  then we  have
   \begin{align*}
    \P \left(  \sum_{j=1}^{rn}  \xi_{j,1}   \xi_{j,2}  \geq   (S_{2})^{2}S_{1} \bigg| \sum_{j=1}^{rn} 1_{\left\{  \xi_{j,1} \geq 1  \right\}}  1_{  \left\{ \xi_{j,2} \geq 1  \right\}} < S_{1} \right) \leq 2s_{1} \P \left(  \xi_{1,1} \geq  S_{2}  \mid  \xi_{1,1} \geq 1  \right)  =   o (\frac{1}{n^2}) \,.
   \end{align*}
   Let $s_{\lambda} > S_{2}^2 S_{1}$. Then by (i),  we deduce
\eqref{eq-Xi-2-bound}, as required.
 \end{proof}

 \begin{proof}[Proof of Lemma \ref{lemma-aux-tree-vs-GW-tree}]
For $u, v \in \mathcal G$, note that on the event $\{|\mathsf{N}_{r}(v)| \leq  \frac{1}{2}n^{\frac{1-\delta}{2}} , \Xi_{2}(u, v) \leq s_{\lambda}\}$, we have that
 $|\mathcal{T}_{\mathrm{aux}} (v)|_{r}| $ is stochastically dominated by $\frac{1}{2}n^{\frac{1-\delta}{2}}+ \sum_{i=1}^{s_{\lambda}} Z^{(i)}_{\leq r}$ where $Z^{(i)}_{\leq r}$ are independent and have the same distribution as $Z_{\leq r}$ (which, as in the proof of Lemma~\ref{lemma-complexity-of-Nr}, is the number of vertices in the first $r$-levels of a $\mathrm{Bin}(n, \lambda/n)$-GW tree). Thus  by Lemmas \ref{lemma-complexity-of-Nr},   \ref{moments}  (and Remark \ref{rmk-moments})
 \begin{equation*}
 \P( |\mathcal{T}_{\mathrm{aux}} (v)|_{r}| > n^{\frac{1-\delta}{2}} )
 \leq  s_{\lambda}  \P ( Z_{\leq r} >   n^{\frac{1}{2}-\delta}   ) +   o(\frac{1}{n^2})   =o(\frac{1}{n^2}) \,.
 \end{equation*}
By symmetry, one can get the same bound for $u$ and this proves the first assertion for Lemma~\ref{lemma-aux-tree-vs-GW-tree}.

We next prove the second assertion, which is  similar to \cite[Lemma 2.2]{RW10}. 
We provide a complete proof here since our auxiliary tree is defined in a slightly non-standard manner.
For  two rooted trees $\tau, \tau'$ with heights at most $r$ and  with $|\tau|, |\tau| \leq   n^{\frac{1-\delta}{2}}$, let  $(b_1, \ldots, b_{|\tau|_{r-1}|})$ and $(b'_1, \ldots, b'_{|\tau'|_{r-1}|})$ be the number of children for vertices obtained along with the breadth-first-search process for $\tau|_{r-1}$ and $\tau'|_{r-1}$ respectively (we do not care about vertices in the $r$-th level as they are surely leaves).
For notation convenience, in this proof we write $t = |\tau|_{r-1}|$ and $t' = |\tau'|_{r-1}|$.
 Note that $\sum_{j=1}^t b_{j}= |\tau| -1$ (and similarly for the prime version). When $\mathcal{T}_{\mathrm{aux}}(u)|_r=\tau$,
 we define by $\sigma$ the map such that $\sigma(\mathbf{i}) $ is the label of the vertex on $\mathcal{T}_{\mathrm{aux}}(u)|_r$ corresponding to $\mathbf{i}$ for $ \mathbf{i} \in \tau$. Then we can regard $\mathcal{T}_{\mathrm{aux}}(u)|_r$ as a labeled rooted ordered tree, and write $\mathcal{T}_{\mathrm{aux}}(u)|_r= (\tau,\sigma)$.
   Therefore, we have
    \begin{equation*}
  \P(  (\mathcal{T}_{\mathrm{aux}}(u))|_r = \tau , (\mathcal{T}_{\mathrm{aux}}(v))|_r = \tau'  ) = \sum_{\sigma, \sigma'} \P( (\mathcal{T}_{\mathrm{aux}}(u))|_r = (\tau, \sigma) , (\mathcal{T}_{\mathrm{aux}}(v))|_r = (\tau',\sigma')  )
  \end{equation*}
 where the sum is over all possible configurations for $\sigma , \sigma'$.
 Given $(\tau, \sigma)$ and $ (\tau', \sigma')$, when we are exploring the $j$-th vertex in the tree $\tau$ (or $\tau'$),  we know that it has $b_{j}$ (or $b_{j}'$) children
 and does not connect to $n-f_{\sigma,\sigma'}(j)$  (or $n-f'_{\sigma,\sigma'}(j)$) vertices, where $|f_{\sigma,\sigma'}(j)| \leq |\tau |$  and $|f'_{\sigma,\sigma'}(j)| \leq |\tau '|$ for all $j$. Therefore,
    \begin{align*}
      & \P(  (\mathcal{T}_{\mathrm{aux}}(u))|_r = (\tau, \sigma) , (\mathcal{T}_{\mathrm{aux}}(v))|_r = (\tau',\sigma')  ) \\
      &= \prod_{j=1}^{t} \big( \frac{ \lambda}{n}  \big)^{ b_{j}} \big(  1-  \frac{ \lambda}{n}  \big)^{n- f_{\sigma,\sigma'}(j)} \prod_{j=1}^{t'} \big( \frac{ \lambda}{n}  \big)^{ b'_{j}} \big(  1-  \frac{ \lambda}{n}  \big)^{n- f'_{\sigma,\sigma'}(j)} \\
      & \in  \big( \frac{\lambda}{n} \big)^{|\tau|+ |\tau'| -2} \big[  \big(  1-  \frac{ \lambda}{n}  \big)^{n(t + t')} ,  \big(  1-  \frac{ \lambda}{n}  \big)^{(n- |\tau| - |\tau'|)  (t+ t')} \big]\,,
    \end{align*}
where as we will see the upper and lower bounds above are very close to each other.
   In addition, the total number of all the configurations for $(\sigma, \sigma') $ is upper-bounded by
    \begin{equation*}
     \binom{n -2}
       {b_{1}, \ldots, b_{t},n- |\tau| -1}
    \binom{n -2}{
       b'_{1}, \ldots, b'_{t'},n- |\tau'| -1}
    \leq  \frac{n^{|\tau| + |\tau'| -2}}{\prod_{j}b_{j}!  \prod_{j}b'_{j}! }
    \end{equation*}
    and lower-bounded by
    \begin{equation*}
      \binom{n -2}
        {b_{1},   \ldots, b_{t},b'_{1}, \ldots, b'_{t'},n- |\tau| - |\tau'|}
= \frac{(n-2)\cdots[n-( |\tau| + |\tau'|-1)]}{ \prod_{j}b_{j}! \prod_{j}b'_{j}!} \,.
    \end{equation*}

    Furthermore, we can compute the analogous probability regarding to two independent PGW($\lambda$)-trees as follows:
    \begin{align*}
      \P( \mathbf{T}|_r=    \tau )\P( \mathbf{T}|_r   = \tau'   )
     = \lambda^{ |\tau| + |\tau'| -2  } e^{-\lambda ( t + t')}   \prod_{j=1}^{t} \frac{1}{b_j!}  \prod_{j=1}^{t'} \frac{1}{b'_j!} \,.
    \end{align*}
    Therefore, noting that $|\tau|, |\tau'| \leq n^{\frac{1-\delta}{2}}$ and using a straightforward computation, we see that
    the ratio $  \frac{\P( (\mathcal{T}_{\mathrm{aux}}(u))|_r = \tau, (\mathcal{T}_{\mathrm{aux}}(v))|_r = \tau')   }{ \P( \mathbf{T}|_r=    \tau )\P( \mathbf{T}|_r  = \tau'   ) } $ converges to 1 as $n\to \infty$, proving the second assertion as required.  
   \end{proof}

Next we prove Lemma \ref{lemma-condi-isomorphic-prob-aux-tree}. Assume that $\mathsf{N}_{r+1}(u) \sim \mathsf{N}_{r+1}(v) $ and $\phi$ is an isomorphism.
Note that \eqref{eq-Av-onto-Au} implies that  $\mathcal{G}[A_{\leq r}(u)]$ and  $\mathcal{G}[A_{\leq r}(v)]$ are isomorphic to each other. However, when there exist cycles in  $\mathcal{G}[A_{\leq r}(u)]$ and $\mathcal{G}[A_{\leq r}(v)]$, our trees
 $\mathcal{T}_{\mathrm{cut}}(u)|_{r}$ and $\mathcal{T}_{\mathrm{cut}}(u)|_{r}$ (which are spanning trees of $\mathcal{G}[A_{\leq r}(u)]$ and $\mathcal{G}[A_{\leq r}(v)]$ respectively) may not be isomorphic since there are some edges deleted. In this case, the event $\mathcal{T}_{\mathrm{cut}}(u)|_{r}\sim \mathcal{T}_{\mathrm{cut}}(v)|_{r} $ occurs or not depends on the labeling configuration on $\mathcal{G}$.

  \begin{proof}[Proof of Lemma \ref{lemma-condi-isomorphic-prob-aux-tree}]
For fixed $u, v\in V$, we say two graphs $G_{1}$ and $G_{2}$ (on $V$) are equivalent if there exists an isomorphism $\varphi$ from $G_{1}$ to $G_{2}$ such that $\varphi(u)=u$ and $\varphi(v)=v$. We write $[G]$ an equivalent class for this equivalent relation.
We can sample the graph $\mathcal{G}$  as follows: first we sample an equivalent class and then we sample the labels uniformly from all labelings that yield the sampled equivalent class.
Now we have
\begin{align*}
  & \P( \mathcal{T}_{\mathrm{aux}}(u)|_{r} \sim  \mathcal{T}_{\mathrm{aux}}(v)|_{r}  | \Omega_{a} \cap \Omega_{b} )   \\
  &= \sum_{[G] \in \Omega_{a} \cap \Omega_{b} } \P( \mathcal{T}_{\mathrm{aux}}(u)|_{r} \sim  \mathcal{T}_{\mathrm{aux}}(v)|_{r}  |   \mathcal{G} \in [G]  )  \P( \mathcal{G} \in [G] |   \Omega_{a} \cap \Omega_{b}   ) \,.
 \end{align*}
Let $\phi$ be an isomorphism between $\mathsf N_{r+1}(u)$ and $\mathsf N_{r+1}(v)$. Recall our reduced BFS. We say a label configuration is successful, if for each vertex $y $ in $A_{t+1}(u)$ with neighbors $z_1(y), \ldots, z_{\ell(y)}(y)$ in $A_{t}(u)$ such that $z_1(y), \ldots, z_{\ell(y)}(y)$ is increasing in the (prefixed) ordering on $V$, then $\phi(z_1(y)), \ldots, \phi(z_{\ell(y)}(y))$ is also increasing in the ordering on $V$.  Clearly, with a successful labeling configuration we have   $\mathcal{T}_{\mathrm{cut}}(u)|_{r} \sim \mathcal{T}_{\mathrm{cut}}(v)|_{r}$.
When  $\mathrm{Comp}(\mathcal{G}[A_{\leq r}(v)]) \leq s_{\lambda} $,
  we claim that
\begin{equation}\label{eq-number-successful-configurations}
 \# \textrm{successful configurations} \geq \frac{\# \textrm{all configurations}}{ ((s_{\lambda} +1 ) ! )^{s_{\lambda}} }\,.
\end{equation}
In order to see this, we observe that $|Y| \leq s_\lambda$ where $Y = \{y \in A_{\leq  r}(u): \ell(y) \geq 2\}$ are the only vertices we need to investigate in order for the labeling configuration to be successful. In addition, once the labels are fixed elsewhere except at $\phi(z_1(y)), \ldots, \phi(z_{\ell(y)}(y))$ for $y\in Y$, the number of valid completions for the labeling configuration is at most $\prod_{y\in Y} \ell_y$ and out of which at least one of them is successful. Since $\ell(y)\leq  \mathrm{Comp}(\mathcal{G}[A_{\leq r}(v)])+1 \leq  s_{\lambda} + 1$, this implies  \eqref{eq-number-successful-configurations}.

At this point, we note that  the event $\mathcal{T}_{\mathrm{aux}}(u)|_{r} \sim \mathcal{T}_{\mathrm{aux}}(v)|_{r}$ occurs if  the following hold:
\begin{itemize}
\item the labeling configuration is successful;

\item for each of the corresponding pairs of the independent $\mathrm{Bin}(n,\frac{\lambda}{n})$-GW trees we grafted they are isomorphic.
\end{itemize} 
When $\Xi_{2} \leq s_{\lambda}$,  we have grafted at most $s_{\lambda}$ pairs of trees. Combined with \eqref{eq-number-successful-configurations}, it gives that  for all $[G] \in \Omega_{a} \cap \Omega_{b}$
  \begin{equation*}
   \P( \mathcal{T}_{\mathrm{aux}}(u)|_{r} \sim  \mathcal{T}_{\mathrm{aux}}(v)|_{r}  |  \mathcal{G} \in [G]  )
     \geq \frac{ \gamma_{\lambda}^{ s_{\lambda}} }{ ((s_{\lambda} +1 ) ! )^{s_{\lambda}} } \,. \qedhere
  \end{equation*}
  \end{proof}

  \begin{proof}[Proof of Lemma \ref{lemma-size-of-bridging-tree}]
    In the case $\lambda < 1$, the statement follows from a simple tail estimate on subcritical branching process.
  Thus, in what follows we assume that $\lambda \geq 1$. Let $f(n)=n^{\frac{\epsilon_0 \wedge \delta }{9}}$.
  Applying  Lemma \ref{lemma-condi-isomorphic-prob-aux-tree}, we have
    \begin{align}\label{lem-4-12-1}
     & \P( \mathsf{N}_{r+1}(v) \sim \mathsf{N}_{r+1}(u), \mathrm{Comp}(\mathsf{N}_{r}(v)) \leq  s_{\lambda}, \Xi_{2} \leq s_{\lambda}, |\mathcal{T}_{\mathrm{cut}}(u)|_{r}| > f(n))  \nonumber\\
     & \leq C_{\lambda} \P(  \mathcal{T}_{\mathrm{aux}}(v)|_{r}  \sim \mathcal{T}_{\mathrm{aux}}(u)|_{r} ; |\mathcal{T}_{\mathrm{aux}}(v)|_{r} |> f(n)  ) \,.
  \end{align}
 Since $\mathcal{T}_{\mathrm{aux}}(u)$ and $\mathcal{T}_{\mathrm{aux}}(v)$ behaves like independent PGW$(\lambda)$ trees (Lemma \ref{lemma-aux-tree-vs-GW-tree}), we have
  \begin{align*}
  & \P(  \mathcal{T}_{\mathrm{aux}}(v)|_{r}  \sim \mathcal{T}_{\mathrm{aux}}(u)|_{r} \,,  f(n) < |\mathcal{T}_{\mathrm{aux}}(v)|_{r} | \leq n ^{\frac{1-\delta}{2}} ) \\
  &= \sum_{\tau \sim \tau', f(n) < |\tau | \leq n ^{\frac{1-\delta}{2}}   }\P(   \mathcal{T}_{\mathrm{aux}}(v)|_{r}= \tau  ,\mathcal{T}_{\mathrm{aux}}(u)|_{r}=\tau' ) \\
  &\lesssim \sum_{\tau \sim \tau', |\tau|> f(n)}\P(   \mathbf{T} |_{r}= \tau  ,\mathbf{T}'|_{r}=\tau' )   = \P(  \mathbf{T}|_{r}  \sim \mathbf{T}'|_{r} ; |\mathbf{T}|_{r} |> f(n) ) \\
 & \leq  \P(   \mathbf{T}  \sim_{r} \mathbf{T}' , |\mathbf{T}|_r |>f(n) ) + \P(   \mathbf{T}  \sim \mathbf{T}' , |\mathbf{T} |> f(n) )   \,,
  \end{align*}
  where the last inequality follows from $\{\mathbf{T}|_{r}  \sim \mathbf{T}'|_{r}\} \subset \{\mathbf{T}  \sim_{r} \mathbf{T}'\} \cup \{ \mathbf{T}  \sim \mathbf{T}'\}$.
 Applying Lemma \ref{cdtmoment}  with  $m =\frac{10}{\epsilon_0 \wedge \delta} $ and applying Markov's inequality, we have
   \begin{align*}
  \P(   \mathbf{T}  \sim_{r} \mathbf{T}' , |\mathbf{T}|_r |>f(n) )   \leq  \E \left[  \frac{  |Z_{\leq r} (\mathbf{T}) |^{m} } {n^{m(\epsilon_0 \wedge \delta)/9}}  ; \mathbf{T}  \sim_{r} \mathbf{T}'     \right] \lesssim    \frac{r^{2m}}{n^{10/9}}  \alpha_{\lambda}^{r} = o (n^{-2}) \,.
   \end{align*}
  Applying  Lemma \ref{isodecay}, we have
 $
   \P(   \mathbf{T}  \sim \mathbf{T}' , |\mathbf{T} |> f(n)) = o (n^{-2}) $.
Altogether, this implies that
$$ \P(  \mathcal{T}_{\mathrm{aux}}(v)|_{r}  \sim \mathcal{T}_{\mathrm{aux}}(u)|_{r} \,,  f(n) < |\mathcal{T}_{\mathrm{aux}}(v)|_{r} | \leq n ^{\frac{1-\delta}{2}} ) = o(n^{-2})\,.$$
Combined with \eqref{lem-4-12-1} and Lemma~\ref{lemma-aux-tree-vs-GW-tree}, this implies that
$$\P( \mathsf{N}_{r+1}(v) \sim \mathsf{N}_{r+1}(u), \mathrm{Comp}(\mathsf{N}_{r}(v)) \leq  s_{\lambda}, \Xi_{2} \leq s_{\lambda}, |\mathcal{T}_{\mathrm{cut}}(u)|_{r}| > f(n) ) = o(n^{-2})\,.$$
Combined with \eqref{eq-prob-complexity-bound} and \eqref{eq-Xi-2-bound}, this competes the proof of the lemma.
     \end{proof}

\end{document}